\documentclass[a4paper, 12pt,reqno]{amsart}

\usepackage[margin=1in]{geometry, stmaryrd}
\usepackage{dirtytalk}
\usepackage{bigints}
\usepackage{amssymb,latexsym}
\usepackage{graphicx}
\usepackage{float}
\graphicspath{ {./images/} }
\usepackage{mathtools}
\usepackage{color}
\usepackage[hyphens]{url}
\usepackage[T1]{fontenc}
\usepackage{hyperref}
\usepackage{amsmath}
\usepackage{mathdots}
\usepackage{breqn}
\usepackage[toc,page]{appendix}
\usepackage{url}    
\usepackage{breqn}
\usepackage{breakurl}
\usepackage{comment}
\usepackage{thmtools}
\usepackage{thm-restate}
\usepackage{mathrsfs}
\usepackage{fancyhdr} 
\usepackage{placeins}
\usepackage{leftidx}
\usepackage{bbm, dsfont}
\allowdisplaybreaks
\theoremstyle{plain}
\numberwithin{equation}{section}
\newtheorem{thm}{Theorem}[section]
\newtheorem{theorem}[thm]{Theorem}
\newtheorem*{theorem*}{Theorem}
\newtheorem{lemma}[thm]{Lemma}
\newtheorem{corollary}[thm]{Corollary}

\newtheorem{definition}[thm]{Definition}
\newtheorem{proposition}[thm]{Proposition}
\newtheorem{conjecture}[thm]{Conjecture}
\newtheorem{remark}[thm]{Remark}

\DeclareMathOperator{\GL}{GL}
\DeclareMathOperator{\Su}{Supp}
\DeclareMathOperator{\Hom}{Hom}
\DeclareMathOperator{\N}{N}
\DeclareMathOperator{\Mat}{Mat}
\DeclareMathOperator{\Ind}{Ind}

\title{O\lowercase{n} \lowercase{the} L\lowercase{ocal} C\lowercase{onverse} T\lowercase{heorem} \lowercase{for} D\lowercase{epth} $\frac{1}{N}$ S\lowercase{upercuspidal} R\lowercase{epresentations of} $\GL(2N, F)$}
\author{D\lowercase{avid} C. L\lowercase{uo} \lowercase{and}  S\lowercase{haun} S\lowercase{tevens}}
\date{}

\address{School of Mathematics, University of Minnesota, Minneapolis, MN 55455, United States} \email{luo00275@umn.edu}

\address{School of Mathematics, University of East Anglia, Norwich Research Park, Norwich NR4 7TJ, United Kingdom} \email{Shaun.Stevens@uea.ac.uk}


\begin{document}

\begin{abstract}
    In this paper, we use type theory to construct a family of depth $\frac{1}{N}$ minimax supercuspidal representations of $\GL(2N, F)$ which we call \textit{middle supercuspidal representations}. These supercuspidals may be viewed as a natural generalization of simple supercuspidal representations, i.e. those supercuspidals of minimal positive depth. Via explicit computations of twisted gamma factors, we show that middle supercuspidal representations may be uniquely determined through twisting by quasi-characters of $F^{\times}$ and simple supercuspidal representations of $\GL(N, F)$. 
\end{abstract}

\maketitle

\section{Introduction}

Let $\pi$ and $\tau$ be irreducible generic representations of the general linear groups $\GL(n, F)$ and $\GL(r, F)$ respectively with $F$ a non-archimedean local field. To $\pi$ and $\tau$, one may attach an invariant called the $\textit{twisted gamma factor}$  $\gamma(s, \pi \times \tau, \psi_{F})$ which is a function of a complex variable $s$ and may be defined using the Rankin--Selberg convolution integrals \cite{JPSS} or the Langlands--Shahidi method \cite{Shahidi}.

Fixing $\pi$ and letting $\tau$ range throughout all irreducible generic representations of $\GL(r, F)$ with $r \geq 1$, one may completely determine $\pi$ via its twisted gamma factors \cite{Chai}, \cite{Liu}; indeed, there is an integer~$M_\pi$ such that it is enough to consider~$1\le r\le M_\pi$, and we call the least such~$M_\pi$ the number of "twists" required to characterize~$\pi$. The question of finding the minimal number of twists needed has a long and rich history (for more information on this, see \cite{Adrian1}). It was conjectured by Herv\'{e} Jacquet that $M_\pi \le \left\lfloor \frac{n}{2} \right\rfloor$; this conjecture is commonly referred to as Jacquet's conjecture or the local converse problem for $\GL(n, F)$ \cite{Jiang}. In 2018, Jacquet and Liu \cite{Liu} and, independently Chai using a different method \cite{Chai}, proved Jacquet's conjecture: 
\begin{theorem*}[Local converse theorem]
    Let $\pi_{1}$ and $\pi_{2}$ be irreducible supercuspidal representations of $\GL(n, F)$. 
     If
    \[
        \gamma(s, \pi_{1} \times \tau, \psi_{F}) = \gamma(s, \pi_{2} \times \tau, \psi_{F})
    \]
    as functions of the complex variable $s$, for all irreducible supercuspidal representations $\tau$ of $\GL(r, F)$ with $1 \leq r \leq \left\lfloor \frac{n}{2} \right\rfloor$, then $\pi_{1} \cong \pi_{2}$.
\end{theorem*}

Following the proof of Jacquet's conjecture, the sharpness of the bound $\left\lfloor \frac{n}{2} \right\rfloor$ has been a topic of much curiosity. While the bound $\left\lfloor \frac{n}{2} \right\rfloor$ is essentially sharp in general \cite{A23}, \cite{Adrian1}, there do exist cases in which it is not. In the situation of depth zero supercuspidal representations~$\pi$, Nien, Zhang, and Yun give precise criteria for when single twists are sufficient \cite{NZ} (i.e.~$M_\pi=1$). Another class of supercuspidal representations that have been proven to satisfy this single twisting condition are \textit{simple (epipelagic) supercuspidal representations}, those supercuspials of minimal positive depth (\cite[Remark 3.18]{Adrian-Liu}, \cite[Proposition 2.2]{BH-1}, and \cite[Theorem 1.1]{Xu}).

In this paper, we consider a class of representations of the next smallest possible positive depth: depth $\frac{1}{N}$ supercuspidal representations of $\GL(2N, F)$, which we call \textit{middle supercuspidal representations}. More precisely, these are the supercuspidal representations of depth $\frac{1}{N}$ such that the field extension~$E/F$ associated to the type from which they are induced has ramification index~$N$ and residue degree~$2$. (More precisely, we should say ``a'' field since, when~$p\mid N$, it may not be unique; however, the ramification index and residue degree are well-defined.) Alternatively, they are the depth $\frac{1}{N}$ supercuspidal representations which are \textit{minimax} in the sense of \cite{Adrian2}; or, in terms of the Galois side of the local Langlands correspondence, they are the depth $\frac{1}{N}$ supercuspidal representations whose Galois parameter has multiplicity-free restriction to wild inertia and twisting number two (i.e. there are precisely two unramified characters of~$F^\times$ which leave the parameter invariant by twisting).


We parametrize middle supercuspidal representations by triples $\left(\bar{f},  \, \chi,  \, \zeta \right)$ where $\bar{f}$ is a monic irreducible degree two polynomial over the residue field $k_{F}$ of $F$, $\chi$ a multiplicative character of the unique degree two extension $k_{\bar{f}}$ of $k_{F}$, and $\zeta \in \mathbb{C}^{\times}$. The first major result we obtain is the following:
\begin{thm}\label{TheoremMain}
    Let $\pi=\pi_{(\bar{f},  \, \chi,  \, \zeta)}$ be a middle supercuspidal representation of $\GL(2N, F)$ and $\pi'$ be an irreducible supercuspidal representation of $\GL(2N, F)$ such that their central characters are equal. If
    \[
        \gamma(s, \pi \times \tau, \psi_{F}) = \gamma(s, \pi' \times \tau, \psi_{F})
    \]
    as functions of the complex variable $s$ for $\tau$ any tamely ramified quasi-character of $F^{\times}$ or any simple supercuspidal representation of $\GL(N, F)$, then $\pi \cong \pi'$.
\end{thm}
\noindent In Theorem \ref{TheoremMain}, we assume that $\pi$ and $\pi'$ share the same central character. If we do not assume this condition, we see from Corollary \ref{CorollaryCentral} below that we only need to add that the gamma factors of $\pi$ and $\pi'$ to agree when twisted by elements of a certain finite set $\Xi_{\text{middle}}$ of positive depth characters of~$F^\times$ that we construct in (\ref{Quasi-Characters}).

Indeed, this is a special case of a more general result. For arbitrary depth $d$, we may define an analogous finite set $\Xi_{d}$ and obtain the following proposition, which may be viewed as an improvement of \cite[Corollary 2.7]{Jiang}:
\begin{proposition}\label{PropositionCentral}
    Let $\pi_{1}$ and $\pi_{2}$ be irreducible supercuspidal representations of $\GL(n, F)$ with central characters $\omega_{1}$ and $\omega_{2}$ respectively, and set~$d=d\left(\pi_{1}\right)$. If 
    \[
    \gamma\left(s, \pi_{1} \times \chi, \psi_{F}\right) = \gamma\left(s, \pi_{2} \times \chi, \psi_{F}\right) 
    \]
    for all $\chi \in \Xi_{d}$, then $\omega_{1} = \omega_{2}$.
\end{proposition}

We remark that in Theorem \ref{TheoremMain}, our choice to twist by simple supercuspidal representations of $\GL(N, F)$ is motivated by \cite{Deligne} which implies the following: let $\pi_{1}$ and $\pi_{2}$ be irreducible supercuspidal representations of $\GL(n, F)$ such that their gamma factors agree for all twists by irreducible supercuspidal representations $\tau$ of $\GL(r, F)$, $1 \leq r \leq \left\lfloor \frac{n}{2} \right\rfloor$. If $d(\tau) > 2 d(\pi_{i})$ (we have $d(\pi_{1}) = d(\pi_{2})$ from Proposition \ref{Xu} as their standard gamma factors agree), then equality of the twisted gamma factors $\gamma(s, \pi_{i} \times \tau, \psi_{F})$ only gives relations between the central characters of the $\pi_{i}$. 

Theorem \ref{TheoremMain} improves the local converse theorem for middle supercuspidal representations by reducing the number of classes of supercuspidal representations $\tau$ of $\GL(r, F)$, $1 \leq r \leq N$ from which to twist by significantly. Furthermore, our theorem gives insight into a potential refinement of the local converse theorem for general supercuspidal representations which we state as the following conjecture:
\begin{conjecture}\label{Conjecture}
            Let $\pi_{1}$ and $\pi_{2}$ be two irreducible supercuspidal representations of $\GL(n, F)$ with the same central character. If 
            \[
            \gamma(s, \pi_{1} \times \tau, \psi_{F}) = \gamma(s, \pi_{2} \times \tau, \psi_{F})
            \]
            for all irreducible supercuspidal representations $\tau$ of $\GL(r, F)$ with $1 \leq r \leq \left\lfloor\frac{n}{2} \right\rfloor$ and $d(\tau) \leq d(\pi_{i})$, then $\pi_{1} \cong \pi_{2}$.
        \end{conjecture}
\noindent From \cite{Deligne}, an analog of our conjecture is true if we instead assume $d(\tau) \leq 2  d(\pi_{i})$. Furthermore, Conjecture \ref{Conjecture} is true for depth zero, simple \cite[Remark 3.18]{Adrian-Liu}, and middle supercuspidal representations.

A future generalization to our work could be refining the local converse theorem for depth $\frac{1}{N}$ minimax supercuspidal representations of $\GL(dN, F)$, with $d$ a positive integer: the case $d=1$ gives simple supercuspidal representations and the case $d=2$ gives middle supercuspidal representations. These supercuspidal representations would be constructed from the simple strata on the standard hereditary $\mathcal{O}_{F}$-order corresponding to the partition $d + \cdots + d$ of $dN$, and an element generating a field extension of $F$ with ramification index $N$ and inertial degree $d$. The methods used here should allow for similar results to be proved for this family of representations.

The organization of the paper is given as follows. In section \ref{PreliminariesAndBackground}, we give a brief account of maximal simple types and the Bushnell--Kutzko construction of supercuspidal representations of $\GL(n, F)$. Next, we discuss an explicit Whittaker function given by Pa\v{s}k\={u}nas and the second-named author which will enable us to compute the twisted gamma factors that we will define via the Rankin--Selberg convolution integrals. In section \ref{SectionDistinguishing}, we construct middle supercuspidal representations of $\GL(2N, F)$ and compute their twisted gamma factors with quasi-characters of $F^{\times}$ and simple supercuspidal representations of $\GL(N, F)$. In doing so, we are able to distinguish middle supercuspidals from each other. Finally in section \ref{SectionAlternate}, we distinguish a given middle supercuspidal representation from other depth $\frac{1}{N}$ supercuspidals of $\GL(2N, F)$ and give a proof of Theorem \ref{TheoremMain}. 

\section*{Acknowledgments}

The authors thank Sarbartha Bhattacharya, Dihua Jiang, Zhaolin Li, Sagnik Mukherjee, and Rongqing Ye for helpful discussions. The first-named author thanks the University of East Anglia for sponsoring his visit to the second-named author. Lastly, the second-named author was supported by EPSRC grant EP/V061739/1.

\section{Preliminaries and background}\label{PreliminariesAndBackground}

In this section, we give the necessary preliminary and background information needed for the remainder of this article.

\subsection{Notation}
To fix some notation, let $F$ be a non-archimedean local field with valuation ring $\mathcal{O}_{F}$, maximal ideal $\mathcal{P}_{F}$ in $\mathcal{O}_{F}$, residue field $k_{F}$ with cardinality $q_{F}$, and $\varpi_{F}$ a fixed uniformizer. Furthermore, let $\psi_{F}$ be a non-trivial additive quasi-character of $F$ with conductor $\mathcal{P}_{F}$, i.e. $\psi_{F}$ is non-trivial on $\mathcal{O}_{F}$ but trivial on $\mathcal{P}_{F}$. For any finite extension $E/F$, we write $e(E|F)$ and $f(E|F)$ for the ramification index and inertial degree of $E/F$ respectively. 

Let $\N(n, F)$ denote the unipotent radical of the standard Borel subgroup of invertible upper triangular matrices of $\GL(n, F)$ and $\text{P}(n, F)$ the standard mirabolic subgroup of $\GL(n, F)$. Furthermore, let $\Mat(n \times m, F)$ denote the space of $n \times m$ matrices with coefficients in $F$ and $I_{n}$ the $n \times n$ identity matrix of $\GL(n, F)$. Next, we let $\psi_{n}$ denote the standard smooth non-degenerate character of $\N(n, F)$ defined by
\[
\psi_{n}(u) = \sum_{i=1}^{n-1}u_{i, i+1}
\]
where $u = (u_{i, j}) \in \N(n, F)$. 

We will fix throughout the article a self-dual Haar measure on $F$, relative to $\psi_{F}$. In particular, this means $\int_{\mathcal{O}_{F}}  dx = q_{F}^{1/2}$. We also fix a Haar measure $d^{*}x$ on $F^{\times}$ such that $\int_{\mathcal{O}_{F}^{\times}}  d^{*}x = 1$. Lastly, let $\Ind$ denote smooth induction with compact induction being denoted by $c$-$\Ind$. 

\subsection{Maximal simple types}\label{MaximalSimpleTypes}

From \cite{BK}, we have that supercuspidal representations of $\GL(n, F)$ may be classified in terms of \textit{maximal simple types} which are special pairs $(J,  \lambda)$ where $J$ is a compact open subgroup of $\GL(n, F)$
and $\lambda$ is an irreducible representation of $J$. We refer to \cite{BK} for precise definitions, and results, the objects introduced in this subsection.

Let $V=F^{n}$ be an $n$-dimensional vector space over $F$ with standard basis. Thus we identify $\text{Aut}_{F}(V)$ with $\GL(n, F)$  and $A = \text{End}_{F}(V)$ with $\Mat(n \times n, F)$. Let $\mathfrak{A}$ be a principal hereditary $\mathcal{O}_{F}$-order in $A$ with Jacobson radical $\mathfrak{P}$. Define
\[
U^{0}\left(\mathfrak{A}\right) = U\left(\mathfrak{A}\right) = \mathfrak{A}^{\times}, \; U^{m}\left(\mathfrak{A}\right) = I_{n} + \mathfrak{P}^{m}, \; m \geq 1.
\]
For $m \geq 0$, choose $\beta \in A$ such that $\beta \in \mathfrak{P}^{-m} \setminus \mathfrak{P}^{1-m}$ where $E = F \left[\beta\right]$ is a field extension of $F$ and $E^{\times}$ normalizes $\mathfrak{A}$. Provided an additional technical condition is satisfied (namely that the critical exponent $k_{0}\left(\beta, \mathfrak{A}\right) < 0$), these data give a principal simple stratum $\left[\mathfrak{A}, m, 0, \beta \right]$ of $A$. Take $J = J\left(\beta,\mathfrak{A}\right)$, $J^{1} = J^{1}\left(\beta,\mathfrak{A}\right)$, and $H^{1} = H^{1}\left(\beta,\mathfrak{A}\right)$. Denote by $\mathcal{C}\left(\mathfrak{A}, \beta, \psi_{F}\right)$ the set of simple (linear) characters of $H^{1}$.

Recall the following definition of maximal simple types.
\begin{definition}
    \normalfont
The pair $(J,  \lambda)$ is called a \textit{maximal simple type} if one of the following holds:
\begin{enumerate}
    \item $J = J\left(\beta,\mathfrak{A}\right)$ is an open compact subgroup associated to a simple stratum $\left[\mathfrak{A}, m, 0, \beta \right]$ of $A$ as above, such that, if $E = F \left[\beta\right]$ and $B = \text{End}_{E}(V)$, then $\mathfrak{B} = \mathfrak{A} \cap B$  is a maximal $\mathcal{O}_{E}$-order in $B$. Moreover, there exists a simple character $\theta \in \mathcal{C}\left(\mathfrak{A}, \beta, \psi_{F}\right)$ such that
    \[
    \lambda \cong \kappa \otimes \sigma
    \]
    where $\kappa$ is a $\beta$-extension of the unique irreducible representation $\eta$ of $J^{1} = J^{1}\left(\beta,\mathfrak{A}\right)$, which contains $\theta$, and $\sigma$ is the inflation to $J$ of an irreducible cuspidal representation of
    \[
    J / J^{1} \cong U\left( \mathfrak{B} \right)/ U^{1}\left( \mathfrak{B} \right) \cong \GL(r, k_{E}),
    \]
    where $r =\frac{n}{[E:F]}$.
    
    \item $(J,  \lambda) = \left(U\left( \mathfrak{A} \right),  \sigma \right)$, where $\mathfrak{A}$ is a maximal hereditary $\mathcal{O}_{F}$-order in $A$ and $\sigma$ is the inflation to $U\left( \mathfrak{A} \right)$ of an irreducible cuspidal representation of
    \[
    U\left( \mathfrak{A} \right)/ U^{1}\left( \mathfrak{A} \right) \cong \GL(n, k_{F})
    \]
\end{enumerate}
We will regard case (2) formally as a special case of case (1) by setting $\beta = 0$, $E=F$, and $\theta$, $\eta$, $\kappa$ trivial. In either case, we write $\textbf{J} = E^{\times}J$. With these data, any irreducible supercuspidal representation $\pi$ of $\GL(n, F)$ is of the form
\[
\pi \cong c\text{-}\Ind_{\textbf{J}}^{\GL(n, F)}\Lambda
\]
for some choice of $(\textbf{J},  \Lambda)$, where $\Lambda \big|_{J} = \lambda$. We call such a pair $(\textbf{J},  \Lambda)$ an \textit{extended maximal simple type}. Let $d(\pi) = \frac{m}{e(E|F)}$ denote the \textit{depth} of $\pi$.
\end{definition}
\noindent For $\pi$ an irreducible supercuspidal representation of $\GL(n, F)$, any two extended maximal simple types in $\pi$ are conjugate in $\GL(n, F)$.

Finally, we recall that for a positive depth supercuspidal representation $\pi$, its associated simple stratum $\left[\mathfrak{A}, m, 0, \beta\right]$ has a \textit{characteristic polynomial} $\phi_{Y_{\beta}}(X) \in k_{F}[X]$ that depends only on the equivalence class of the final simple stratum in the defining sequence of $\left[\mathfrak{A}, m, 0, \beta\right]$ \cite[Chapter 2, p. 58]{BK}. 
\begin{definition}
    \normalfont
    Let $e := e\left(\mathfrak{A}\right)$ and $g := \gcd(e, m)$. We set 
    \[
    Y_{\beta} = \beta^{e/g}\varpi_{F}^{m/g} \in A.
    \]
    Let $\Phi_{Y_{\beta}}(X) \in F[X]$ be the characteristic polynomial of $Y_{\beta}$ as an $F$-endomorphism of $V$. Then, since $Y_{\beta} \in \mathfrak{A}$, we have $\Phi_{Y_{\beta}}(X) \in \mathcal{O}_{F}[X]$. We call the reduction $\phi_{Y_{\beta}}(X)$ of $\Phi_{Y_{\beta}}(X)$ modulo $\mathcal{P}_{F}$ the \textit{characteristic polynomial} of $\left[\mathfrak{A}, m, 0, \beta\right]$. Furthermore, $\phi_{Y_{\beta}}$ is a power of a monic irreducible polynomial $\phi_{\pi}(X) \in k_{F}[X]$ that we call the \textit{minimal polynomial} of $\left[\mathfrak{A}, m, 0, \beta\right]$. From \cite[Subsection 6.1]{BHK}, we have that $\phi_{\pi}$ is an invariant of $\pi$ with $\phi_{\pi}(X) \neq X$. 
\end{definition}

\subsection{Explicit Whittaker functions}\label{ExplicitWhittakerfunctions}

In this subsection, we introduce an explicit Whittaker function $\mathcal{W}_{\pi}$ in the \textit{Whittaker model} $W\left(\pi, \psi_{F}\right)$ associated with a supercuspidal representation $\pi$ of $\GL(n, F)$ \cite{Paskunas}. To do so, we first introduce Bessel functions of irreducible supercuspidal representations of $\GL(n, F)$ constructed by Pa\v{s}k\={u}nas and the second named author. We recall the basics of these Bessel functions, which rely on the construction theory of supercuspidal representations of $\GL(n, F)$ in terms of maximal simple types of Bushnell and Kutzko introduced in subsection \ref{MaximalSimpleTypes}. We will use the notation from \cite{BK} and \cite{Paskunas}.

We recall the following (\cite[Section 2]{Adrian-Liu} and \cite{BH-2}): an irreducible admissible representation $\left(\pi, V_{\pi}\right)$ of $\GL(n, F)$ is called \textit{generic} if 
\[
\Hom_{\GL(n, F)}\left(\pi,  \Ind_{\N(n, F)}^{\GL(n, F)}\psi_{n}\right) \neq 0.
\]
By the uniqueness of local Whittaker models, this $\Hom$-space is at most one-dimensional. By Frobenius reciprocity,
\[
\Hom_{\GL(n, F)}\left(\pi,  \Ind_{\N(n, F)}^{\GL(n, F)}\psi_{n}\right) \cong \Hom_{\N(n, F)}\left(\pi\big|_{\N(n, F)},  \psi_{n}\right).
\]
Therefore, $\Hom_{\N(n, F)}\left(\pi \big|_{\N(n, F)},  \psi_{n}\right)$ is also at most one dimensional. Assume that $\left(\pi, V_{\pi}\right)$ is generic. Fix a non-zero functional $l \in \Hom_{\N(n, F)}\left(\pi \big|_{\N(n, F)},  \psi_{n}\right)$, which is unique up to scalar. The Whittaker function attached to a vector $v \in V_{\pi}$ is defined by
\[
W_{v}(g) := l \left(\pi(g)v\right), \; \text{ for all } g \in \GL(n, F),
\]
so that $W_{v} \in \Ind_{\N(n, F)}^{\GL(n, F)}\psi_{n}$. The space
\[
W(\pi, \psi_{n}) := \{ W_{v} : v \in V_{\pi}\}
\]
is called the \textit{Whittaker model} of $\pi$, and $\GL(n, F)$ acts on it by right translation which we denote by $R$. It is easy to see that the Whittaker model of $\pi$ is independent of the choice of the non-zero functional $l$.

Next, we recall from \cite[Section 5]{Paskunas} the general formulation of Bessel functions. Let $\mathcal{K}$ be an open compact-modulo-centre subgroup of  $\GL(n, F)$ and let $\mathcal{U} \subset \mathcal{M} \subset \mathcal{K}$ be compact open subgroups of $\mathcal{K}$. Let $\tau$ be an irreducible smooth representation of $\mathcal{K}$ and let $\Psi$ be a linear character of $\mathcal{U}$. Take an open normal subgroup $\mathcal{N}$ of $\mathcal{K}$, which is contained in $\text{Ker}(\tau) \cap \mathcal{U}$. Let $\chi_{\tau}$ be the (trace) character of $\tau$. The associated Bessel function $\mathcal{J} \, \colon \, \mathcal{K}  \to \mathbb{C}$ of $\tau$ is defined by
\[
\mathcal{J}(g) := \left[\mathcal{U} : \mathcal{N}\right]^{-1}\displaystyle\sum_{h \in \mathcal{U}/\mathcal{N}}\Psi(h)^{-1}\chi_{\tau}(gh).
\]
This is independent of the choice of $\mathcal{N}$. The basic properties of this Bessel function which we will need are given below.
\begin{proposition}\cite[Proposition 5.3]{Paskunas}\label{Paskunas5.3}
Assume that the data introduced above satisfy the following:
\begin{itemize}
    \item $\tau \big|_{\mathcal{M}}$ is an irreducible representation of $\mathcal{M}$;
    \item $\tau \big|_{\mathcal{M}} \cong \Ind_{\mathcal{U}}^{\mathcal{M}}\left(\Psi\right)$.
\end{itemize}
Then the Bessel function $\mathcal{J}$ of $\tau$ enjoys the following properties:
\begin{enumerate}
    \item $\mathcal{J}(1)=1$;
    \item $\mathcal{J}(hg) = \mathcal{J}(gh) = \Psi(h)\mathcal{J}(g)$ for all $h \in \mathcal{U}$ and $g \in \mathcal{K}$;
    \item if $\mathcal{J}(g) \neq 0$, then $g$ intertwines $\Psi$; in particular, if $m \in \mathcal{M}$, then $\mathcal{J}(m) \neq 0$ if and only if $m \in \mathcal{U}$.
\end{enumerate}
\end{proposition}

By \cite[Proposition 1.6]{BH-2}, there is an extended maximal simple type $(\textbf{J}, \Lambda)$ in $\pi$ such that
\[
\Hom_{\N(n, F)  \cap  \textbf{J}}\left(\psi_{n}, \Lambda\right) \neq 0.
\]
Since $\Lambda$ restricts to a multiple of some simple character $\theta \in \mathcal{C}\left(\mathfrak{A}, \beta, \psi_{F}\right)$, one finds that $\theta(u) = \psi_{u}(h)$ for all $u \in \N(n, F) \cap H^{1}$. As in \cite[Definition 4.2]{Paskunas}, one defines a character $\Psi_{n} \, \colon \, (J \cap \N(n, F))  H^{1} \to \mathbb{C}^{\times}$ by
\[
\Psi_{n}(uh) := \psi_{n}(u)\theta(h)
\]
for all $u \in J \cap \N(n, F)$ and $h \in H^{1}$. By \cite[Theorem 4.4]{Paskunas}, the data
\[
\mathcal{K} = \textbf{J}, \; \tau = \Lambda, \; \mathcal{M} = (J \cap P(n, F))J^{1}, \; \mathcal{U} = (J \cap \N(n, F))H^{1}, \text{ and } \Psi = \Psi_{n}
\]
satisfy the conditions in Proposition \ref{Paskunas5.3} and hence define a Bessel function $\mathcal{J}$.

Now we define a function $\mathcal{W}_{\pi} \, \colon \, \GL(n, F) \to \mathbb{C}$ by
\begin{equation}\label{Whittaker}
    \mathcal{W}_{\pi}:= \begin{cases}
        \psi_{n}(u)\mathcal{J}(j) \; &\text{ if } g=uj \text{ with } u \in \N(n, F), \; j \in \textbf{J}, \\
        0 \; &\text{ otherwise,}
    \end{cases}
\end{equation}
which is well-defined by Proposition \ref{Paskunas5.3} (2). Then, by \cite[Theorem 5.8]{Paskunas}, $\mathcal{W}_{\pi} \in W(\pi, \psi_{F})$ is a Whittaker function for $\pi$. Lastly, we will mention a useful property about the support $\Su\left(\mathcal{W}_{\pi}\right)$ of $\mathcal{W}_{\pi}$ that we will refer to in our calculations:
\begin{thm}
\normalfont
    $\Su\left(\mathcal{W}_{\pi}\right) \cap \text{P}(n, F) = \text{N}(n, F)\left(H^{1} \cap \text{P}(n, F)\right)$ and 
    \begin{equation*}
        \mathcal{W}_{\pi}(um) = \psi_{n}(u)\psi_{\beta}(m)
    \end{equation*}
    for all $u \in \text{N}(n, F)$ and $ m \in H^{1}\cap \text{P}(n, F)$.
\end{thm}

\subsection{Simple supercuspidal representations}\label{SimpleSupercuspidalRepresentations}

In this subsection, we briefly recall the construction of \textit{simple supercuspidal representations} of $\GL(n, F)$ via type theory (\cite{Knightly} and \cite[Example 2.18]{Ye}). Let $\mathfrak{I}_{n}$ be the standard minimal hereditary $\mathcal{O}_{F}$-order in $\Mat(n \times n, F)$ and $\mathfrak{P}_{\mathfrak{I}_{n}}$ the corresponding Jacobson radical. We consider the simple stratum $\left[\mathfrak{I}_{n}, 1, 0, \beta_{u}\right]$ where
\[
\beta_{u} = \begin{pmatrix}
    & & & & (u\varpi_{F})^{-1} \\
    1 & & & & \\
     & 1 & & & \\
      & & \ddots & & \\
      & & & 1 &
\end{pmatrix}
\]
is minimal over $F$ and $E_{u} = F \left[  \beta_{u} \right]$ is a degree $n$ totally ramified extension of $F$.

Associated with the simple stratum above are the open normal subgroups of $\mathfrak{I}_{n}^{\times}$:
\[
H^{1}\left(\beta_{u}, \mathfrak{I}_{n}\right) = J^{1}\left(\beta_{u}, \mathfrak{I}_{n}\right) = U^{1}\left(\mathfrak{I}_{n}\right) = I_{n} + \mathfrak{P}_{\mathfrak{I}_{n}}
\]
and the open compact subgroup $J_{n} := J\left(\beta_{u}, \mathfrak{I}_{n}\right) = \mathcal{O}_{F}^{\times} \, U^{1}\left( \mathfrak{I}_{n}\right)$. From the singleton $\mathcal{C}\left(\mathfrak{I}_{n}, \beta_{u}, \psi_{F} \right) = \{\psi_{\beta_{u}}\}$, we may form a $\beta_{u}$-extension $\kappa_{(u,  \, \phi)} \, \colon \, \mathcal{O}_{F}^{\times}  \, U^{1}\left(\mathfrak{I}_{n}\right) \to \mathbb{C}^{\times}$ of $\psi_{\beta_{u}}$ by
\[
\kappa_{(u,  \, \phi)}(xy) := \phi(x)\psi_{\beta_{u}}(y)
\]
where $x \in \mathcal{O}_{F}^{\times}$ and $y \in U^{1}\left(\mathfrak{I}_{n}\right)$ with $\phi$ trivial on $1 + \mathcal{P}_{F}$, i.e. $\phi$ is inflated from a quasi-character of $k_{F}^{\times}$.

To form a maximal simple type with the data above, we take the pair $\left(J_{n},  \kappa_{(u,  \, \phi)}\right)$. Moreover, to form an extended maximal simple type $\left(\textbf{J}_{u},  \Lambda_{(u,  \, \phi,  \, \zeta)}\right)$, let $\Lambda_{(u,  \, \phi,  \, \zeta)}\left(\beta_{u}\right) = \zeta \in \mathbb{C}^{\times}$ so that $\Lambda_{(u,  \, \phi,  \, \zeta)}(\varpi_{F}) = \left(\zeta^{n}  \phi(u)\right)^{-1}$. Hence, we may now define a simple supercuspidal representation:
\begin{definition}
    \normalfont 
    Let $\left(J_{u},  \kappa_{(u,  \, \phi)}\right)$ be a maximal simple type as defined above. A \textit{simple supercuspidal representation} of $\GL(n, F)$ is a depth $\frac{1}{n}$ supercuspidal representation of the form 
    \[
    \pi_{(u,  \, \phi,  \, \zeta)} = c\text{-}\Ind_{\textbf{J}_{u}}^{\GL(n, F)}\Lambda_{(u,  \, \phi,  \, \zeta)}.
    \]
    Let $\omega_{(u,  \, \phi,  \, \zeta)}$ denote the central character of $\pi_{(u,  \, \phi,  \, \zeta)}$. 
\end{definition}

\begin{remark}
\normalfont
    We parametrize simple supercuspidal representations here via $(u,  \, \phi,  \, \zeta)$ as there exists a bijection between the set of isomorphism classes of simple supercuspidal representations of $\GL(n, F)$ and the set of all such triples \cite[Proposition 1.3]{Imai}. 
\end{remark}

We end this subsection with an explicit Whittaker function $\mathcal{W}_{(u,  \, \phi,  \, \zeta)} \in W\left(\pi_{(u,  \, \phi,  \, \zeta)}, \psi_{F}\right)$ via the following \cite[Section 3.3]{Adrian-Liu}, \cite[Example 2.23]{Ye}:

\hspace{0pt}\resizebox{1.0\linewidth}{!}{
  \begin{minipage}{\linewidth}
\begin{align*}
    g \mapsto  \begin{cases} 
      \hfill \psi_{n}(u)  \Lambda_{\left(u,  \chi,  \zeta\right)}(h) &, \text{ if } g = uh \in \N(n, F) \; \textbf{J}_{u} \text{ with } u \in \N(n, F), \; h \in \textbf{J}_{u} \\
      \hfill 0 & ,  \text{ otherwise}
   \end{cases}.
\end{align*}
\end{minipage}
}

\subsection{Twisted local factors}\label{TwistedLocalFactors}
In this subsection, we recall some basic facts about twisted local factors. Suppose that $\pi_{1}$ is a generic representation of $\GL(n, F)$ and $\pi_{2}$ is a generic representation of $\GL(m, F)$ with $m < n$ and central characters $\omega_{\pi_{1}}$ and $\omega_{\pi_{2}}$ respectively. Furthermore, let 
\[
w_{n, m} = \begin{pmatrix}
    I_{m} & \\
    & w_{n-m}
\end{pmatrix} \in \GL(n, F), \text{ where \; } w_{r} = \begin{pmatrix}
    & & 1 \\
    & \iddots & \\
    1 & &
\end{pmatrix} \in \GL(r, F).
\]

Next, let $W_{\pi_{1}} \in W\left(\pi_{1}, \psi_{n}\right)$ and $W_{\pi_{2}} \in W\left(\pi_{2}, \psi_{m}^{-1}\right)$. We define the \textit{Rankin--Selberg convolution integrals} $\widetilde{\Psi}$ and $\Psi$ attached to $\pi_{1}$ and $\pi_{2}$ by 
\begin{align*}
    & \widetilde{\Psi}\left(s; W_{\pi_{1}}, W_{\pi_{2}}\right) \\
    &= \int\displaylimits_{\N(m, F) \setminus \GL(m, F)}\int\displaylimits_{\Mat(n-m-1 \times m, F)}W_{\pi_{1}}\left(\begin{pmatrix}
        h & & \\
        x & I_{n-m-1} & \\
        & & 1
    \end{pmatrix}\right)W_{\pi_{2}}(h)|\det(h)|^{s-\frac{n-m}{2}} \; dx \; dh,
\end{align*}
and 
\[
\Psi\left(s; W_{\pi_{1}}, W_{\pi_{2}}\right) = \int\displaylimits_{\N(m, F) \setminus \GL(m, F)} W_{\pi_{1}}\left(\begin{pmatrix}
    h & \\
    & I_{n-m}
\end{pmatrix}\right)W_{\pi_{2}}(h)|\det(h)|^{s-\frac{n-m}{2}} \; dh.
\]
These integrals are absolutely convergent for $\text{Re}(s)$ sufficiently
large and are rational functions of $q_{F}^{-s}$ \cite[Section 2.7]{JPSS}.

From these integrals, we obtain the following equation which defines the twisted gamma factor $\gamma(s, \pi_{1} \times \pi_{2}, \psi_{F})$:
\begin{thm}\cite[Section 2.7]{JPSS}\label{GammaFactor2.7}
    There is a rational function $\gamma(s, \pi_{1} \times \pi_{2}, \psi_{F}) \in \mathbb{C}\left(q_{F}^{-s}\right)$ such that
    \[
    \widetilde{\Psi}\left(1-s; \rho(w_{n, m})\widetilde{W}_{\pi_{1}}, \widetilde{W}_{\pi_{2}}\right) = \omega_{\pi_{2}}(-1)^{n-1}\gamma(s, \pi_{1} \times \pi_{2}, \psi_{F})\Psi\left(s; W_{\pi_{1}}, W_{\pi_{2}}\right),
    \]
for all $W_{\pi_{1}} \in W\left(\pi_{1}, \psi_{n}\right)$, $W_{\pi_{2}} \in W\left(\pi_{2}, \psi_{m}^{-1}\right)$, where $\widetilde{W}_{\pi_{1}}(g) = W_{\pi_{1}}\left(w_{n}\leftidx^{t}g^{-1}\right)$ and $\widetilde{W}_{\pi_{2}}(h) = W_{\pi_{2}}\left(w_{m}\leftidx^{t}h^{-1}\right)$, for $g \in \GL(n, F)$ and $h \in \GL(m, F)$.
\end{thm}

\begin{definition}
    \normalfont
    The local factor $\epsilon(s, \pi_{1} \times \pi_{2}, \psi_{F})$ is defined as the ratio
    \[
    \epsilon(s, \pi_{1} \times \pi_{2}, \psi_{F}) = \frac{\gamma(s, \pi_{1} \times \pi_{2}, \psi_{F})L(s, \pi_{1} \times \pi_{2})}{L(1-s, \pi_{1}^{\vee}\times \pi_{2}^{\vee})}
    \]
    where $\pi_{1}^{\vee}$ and $\pi_{2}^{\vee}$ denote the contragredients of $\pi_{1}$ and $\pi_{2}$ respectively. 
\end{definition}

\begin{thm}\cite[Theorem 3.4]{Cogdell0}
    If $\pi_{1}$ and $\pi_{2}$ are both (unitary) supercuspidal representations, then $L(s, \pi_{1} \times \pi_{2}) \equiv 1$, hence, $\epsilon(s, \pi_{1} \times \pi_{2}, \psi_{F}) = \gamma(s, \pi_{1} \times \pi_{2}, \psi_{F})$.
\end{thm}

Finally, we introduce the \textit{conductor} $f\left(\pi_{1} \times \pi_{2}, \psi_{F}\right)$ attached to $\pi_{1}$ and $\pi_{2}$. 
\begin{thm}\label{TheoremConductor}\cite[Section 2.7]{JPSS}\label{Conductor}
    Let $\pi_{1}$ and $\pi_{2}$ be as above. There are $c \in \mathbb{C}^{\times}$ and $f\left(\pi_{1} \times \pi_{2}, \psi_{F}\right) \in \mathbb{Z}$ such that 
    \[
    \gamma(s, \pi_{1} \times \pi_{2}, \psi_{F}) = c \cdot  q_{F}^{-f\left(\pi_{1} \times \pi_{2}, \psi_{F}\right)s}.
    \]
\end{thm}
We further note that $f\left(\pi_{1} \times \pi_{2}, \psi_{F}\right) = f\left(\pi_{1} \times \pi_{2}\right)-nm$ where $f\left(\pi_{1} \times \pi_{2}\right)$ is independent of $\psi_{F}$ \cite{BHK}.

\subsection{Central characters}\label{CentralCharacter}

In this subsection, we show that for two irreducible supercuspidal representations $\pi_{1}$ and $\pi_{2}$ of $\GL(n, F)$, $n \geq 2$ of positive depth $d$, one may choose a finite set $\Xi_{d}$ of quasi-characters of $F^{\times}$ such that if the twisted gamma factors $\gamma\left(s, \pi_{i} \times \chi, \psi_{F}\right)$ with $\chi \in \Xi_{d}$ are equal, then the central characters of $\pi_{1}$ and $\pi_{2}$ are equal. To begin, we recall \cite[Proposition 2.6]{Jiang}:
\begin{proposition}\label{Jiang}
    Let $\pi$ be an irreducible generic representation of $\GL(n, F)$ with $n \geq 2$ and $\omega_{\pi}$ the central character of $\pi$. Then there exists $m_{\pi}$ such that, for any character $\chi$ of $F^{\times}$ of conductor $m \geq m_{\pi}$ and any $c \in \mathcal{P}_{F}^{-m}$ satisfying $\chi(1+x) = \psi_{F}(cx)$ for $x \in \mathcal{P}_{F}^{\lfloor\frac{m}{2}\rfloor+1}$, we have
    \[
    L(s, \pi \times \chi) = 1 \text{ and \,} \gamma(s, \pi \times \chi, \psi_{F}) = \omega_{\pi}(c)^{-1}\gamma(s, \chi, \psi_{F})^{n}.
    \]
\end{proposition}
\noindent From \cite{Harris} and \cite[Lemma 2.2]{Xu}, one may choose any $m_{\pi} > 2 d(\pi)$ in the preceding proposition. Next, we recall \cite[Proposition 2.1]{Xu}:
\begin{proposition}\label{Xu}
    Let $\pi_{1}$ and $\pi_{2}$ be two irreducible supercuspidal representations of $\GL(n, F)$ such that their standard gamma factors $\gamma(s, \pi_{i}, \psi_{F})$ agree. Then $\pi_{1}$ and $\pi_{2}$ have the same depth. 
\end{proposition}

For a supercuspidal representation $\pi$ of depth $d$, its central character $\omega$ is trivial on $U^{\lceil d \rceil}(\mathcal{O}_{F}) = 1 + \mathcal{P}_{F}^{\lceil d \rceil}$ \cite{BK}. Choosing a set of generators $\{r_{i}\}$ of $\mathcal{O}_{F}^{\times}/U^{\lceil d \rceil}(\mathcal{O}_{F})$, we define the finite set $C_{d} := \left\{r_{i}\varpi_{F}^{-\lceil 2d + 1 \rceil}, \varpi_{F}^{-\lceil 2d + 2 \rceil}\right\}$. For each element $r_{i}\varpi_{F}^{-\lceil 2d + 1 \rceil} \in C_{d}$, we choose a quasi-character $\chi_{r_{i}\varpi_{F}^{-\lceil 2d + 1 \rceil}}$ of $F^{\times}$ with conductor $m = \lceil 2d + 1 \rceil > 2d$ such that
\[
\chi_{r_{i}\varpi_{F}^{-\lceil 2d + 1 \rceil}}(1+x) = \psi_{F}\left(r_{i}\varpi_{F}^{-\lceil 2d + 1 \rceil}x\right)
\]
for all $x \in \mathcal{P}_{F}^{\lfloor \frac{m}{2} \rfloor+1}$. Similarly for $\varpi_{F}^{-\lceil 2d + 2 \rceil} \in C_{d}$, we choose a quasi-character $\chi_{\varpi_{F}^{-\lceil 2d + 2 \rceil}}$ of $F^{\times}$ with conductor $m' = \lceil 2d + 2 \rceil > 2d$ such that $\chi_{\varpi_{F}^{-\lceil 2d + 2 \rceil}}(1+x) = \psi_{F}\left(\varpi_{F}^{-\lceil 2d + 2 \rceil}x\right)$
for all $x \in \mathcal{P}_{F}^{\lfloor \frac{m'}{2} \rfloor+1}$.

Hence, we define 
\begin{equation}\label{Xi}
    \Xi_{d} := \left\{\chi_{r_{i}\varpi_{F}^{-\lceil 2d + 1 \rceil}}, \chi_{\varpi_{F}^{-\lceil 2d + 2 \rceil}}, \mathbbm{1}_{F}\right\}
\end{equation}
where $\mathbbm{1}_{F}$ denotes the trivial quasi-character of $F^{\times}$. This enables us to prove Proposition \ref{PropositionCentral}.
\begin{proof}[Proof of Proposition \ref{PropositionCentral}]
     Since the standard gamma factors of $\pi_{1}$ and $\pi_{2}$ agree, they share the same depth from Proposition \ref{Xu}. From the discussion above, we have that the central character $\omega_{i}$ is trivial on $1 + \mathcal{P}_{F}^{\lceil d \rceil}$. The proof of the proposition now follows from Proposition \ref{Jiang} as equality of the twisted gamma factors given in the hypothesis implies that $\omega_{1} = \omega_{2}$ on the set of generators $\{r_{i}\}$ and $\varpi_{F}$.
\end{proof}

\section{Middle supercuspidal representations}\label{SectionDistinguishing}

In this section, we construct depth $\frac{1}{N}$ minimax supercuspidal representations of the group $\GL(2N, F)$, called \textit{middle supercuspidal representations}, and use their twisted gamma factors to distinguish them from each other. Specifically, we will twist by tamely ramified quasi-characters of $F^{\times}$ and simple supercuspidal representations of $\GL(N, F)$. 

\subsection{Constructing middle supercuspidal representations}\label{Middle}

In this subsection, we construct depth $\frac{1}{N}$ minimax supercuspidal representations $\pi$ of $\GL(2N,  F)$ ($N \geq 2$). Let $f = X^{2}-dX-c$ with $c$,  $d \in \mathcal{O}_{F}$ such that $f\mod{\mathcal{P}_{F}}$ is irreducible over $k_{F}$. Furthermore, let $s = (s_{i, j}) \in S_{2N}$ be the permutation matrix corresponding to the permutation 
\begin{equation*}
    \begin{pmatrix}
        1 & 2 & \cdots & N & | & N+1 & N+2 & \cdots 2N \\
        1 & 3 & \cdots & 2N-1 & | & 2 & 4 & \cdots 2N
    \end{pmatrix}
\end{equation*}
and $t_{f} = \left(t_{i, j}\right) \in \GL(2N,  F)$ be the diagonal matrix defined by: $t_{i, i} = 1$ for $1 \leq i \leq N$ and $t_{i, i} = c^{-1}\varpi_{F}$ for $N+1 \leq i \leq 2N$. From $s$ and $t_{f}$, we set $g_{f} := t_{f}s^{-1} \in \GL(2N,  F)$. Let $\mathfrak{A}_{2N}$ be the $g_{f}$-conjugation of the standard hereditary $\mathcal{O}_{F}$-order $\mathfrak{A}_{2}$ in $ A = \Mat(2N \times 2N, F)$ corresponding to the partition of $2N$ using only 2's:
\begin{equation}\label{Order}
    \mathfrak{A}_{2N} = \begin{pmatrix}
        \mathfrak{I}_{N} & \varpi_{F}^{-1}\mathfrak{I}_{N}  \\
        \varpi_{F}\mathfrak{I}_{N} & \mathfrak{I}_{N}  
    \end{pmatrix}
\end{equation}
where $\mathfrak{I}_{N}$ is the minimal standard hereditary $\mathcal{O}_{F}$-order in $\Mat(N \times N, F)$ with Jacobson radical $\mathfrak{P}_{\mathfrak{I}_{N}}$.

Next, let 
\begin{equation}
    \beta_{f} := g_{f}\begin{pmatrix}
 &  &  & & &  \varpi_{F}^{-1} \\
& & & &  c\varpi_{F}^{-1} & d\varpi_{F}^{-1} \\
    1 & & & & & \\
    &  1 & & & & \\
    & & \ddots & & &\\
    & & & 1 &  &
\end{pmatrix}g_{f}^{-1} 
= \begin{pmatrix}
    & & & & & c\varpi_{F}^{-2} \\
    1 & & & & &\\
    & \ddots & & & &  \\
    & & 1 & & & d\varpi_{F}^{-1}\\
    &  & & \ddots  & & \\
    & & & & 1 &
\end{pmatrix}
\end{equation}
where $d\varpi_{F}^{-1} = \left(\beta_{f}\right)_{N+1,  2N}$. 

The element $\beta_{f}$ gives us a degree $2N$ extension $E_{f} = F\left[\beta_{f}\right]$ of $F$ with inertial degree $f(E_{f}|F) = 2$, ramification index $e(E_{f}|F) = N$, and $\nu_{E_{f}}\left(\beta_{f}\right) = -1$ where $\nu_{E_{f}}$ is the normalized discrete valuation on $E_{f}$. Furthermore, we let $\sigma_{f} := \beta_{f}^{N}\varpi_{F}$ be a zero of $X^{2}-dX-c$ and $L_{f} := F\left[\sigma_{f}\right]$ the unique degree two unramified extension of $F$. It is not hard to see that $\left[\mathfrak{A}_{2N}, 1, 0, \beta_{f}\right]$ is a simple stratum with $\beta_{f}$ minimal over $F$ \cite[Chapter 1]{BK}. This implies that the supercuspidal representations coming from our simple stratum will be minimax in the sense of \cite{Adrian2}. 
 
Associated with our simple stratum $\left[\mathfrak{A}_{2N}, 1, 0, \beta_{f}\right]$ is the open compact subgroup $J_{f} := J\left(\beta_{f}, \mathfrak{A}_{2N}\right)$ with open normal subgroups $J^{1}:= J^{1}\left(\beta_{f}, \mathfrak{A}_{2N}\right)$ and $H^{1}:= H^{1}\left(\beta_{f}, \mathfrak{A}_{2N}\right)$. In our situation, we have that
\[
J^{1} = H^{1} = U^{1}:= U^{1}\left(\mathfrak{A}_{2N}\right) = I_{2N} + \mathfrak{P}_{2N}
\]
where $\mathfrak{P}_{2N}$ is the Jacobson radical of $\mathfrak{A}_{2N}$:
\begin{equation}\label{Jacobson}
    \mathfrak{P}_{2N} = \begin{pmatrix}
        \mathfrak{P}_{\mathfrak{I}_{N}} & \varpi_{F}^{-1}\mathfrak{P}_{\mathfrak{I}_{N}}  \\
        \varpi_{F}\mathfrak{P}_{\mathfrak{I}_{N}} & \mathfrak{P}_{\mathfrak{I}_{N}}  
    \end{pmatrix}.
\end{equation}

\noindent Furthermore, $J_{f} = \mathcal{O}_{L_{f}}^{\times} \, U^{1}$ where if we fix the ordered $\mathcal{O}_{F}$-basis $\{c,   \sigma_{f}\}$ of $\mathcal{O}_{L_{f}}$, a general element $a_{0}c+a_{1}\sigma_{f} \in \mathcal{O}_{L_{f}}^{\times}$ is of the form
\begin{equation}\label{decomposition}
\begin{pmatrix}
        a_{0}c & & & & a_{1}c\varpi_{F}^{-1} & & &\\
         & a_{0}c & & & & a_{1}c\varpi_{F}^{-1} & & \\
         & & \ddots & & & & \ddots & \\
         &  & & a_{0}c & & & & a_{1}c\varpi_{F}^{-1} \\
        a_{1}\varpi_{F} & & & & a_{0}c+a_{1}d & & & \\
          & a_{1}\varpi_{F} & & &  & a_{0}c+a_{1}d & & \\
          & & \ddots & &  & & \ddots & \\
         & & &  a_{1}\varpi_{F} & & & &a_{0}c+a_{1}d 
    \end{pmatrix}
\end{equation}
\noindent with $a_{i} \in \mathcal{O}_{F}$. 
 
Next, we note that the set of simple characters $\mathcal{C}\left(\mathfrak{A}_{2N}, \beta_{f}, \psi_{F}\right)$ of $H^{1}$ is the singleton $\{\psi_{\beta_{f}}\}$ where
\[
\psi_{\beta_{f}}(x) = \psi_{F} \circ \text{tr}_{E_{f}/F}(\beta_{f}(x-1))
\]
for $x \in U^{1}$. From this simple character, we may define a $\beta_{f}$-\textit{extension} $\kappa_{(f, \, \chi)}$ \cite[Chapter 5]{BK}. In our situation, there exist $q_{F}^{2}-1$ distinct $\beta_{f}$-extensions which are quasi-characters of $J_{f}$ given by
\[
   \kappa_{(f, \, \chi)}(xy) :=  \chi(x)\psi_{\beta_{f}}(y)
\]
for $x \in \mathcal{O}_{L_{f}}^{\times}$ and $y \in U^{1}$ where $\chi$ is a quasi-character of $\mathcal{O}_{L_{f}}^{\times}$ that is trivial on $1 + \mathcal{P}_{L_{f}}$, i.e. $\chi$ is the inflation of a quasi-character of $k_{L_{f}}^{\times}$ \cite[Theorem 5.2.2]{BK}. It follows trivially that $\kappa_{(f, \, \chi)}$ agrees with the conditions of $\beta_{f}$-extensions given in \cite[Chapter 5]{BK}. 

We may now introduce simple types $\left(J_{f},  \lambda_{(f, \, \chi)}\right)$ associated with our simple stratum. In Bushnell--Kutzko notation, we remark that $\lambda_{(f, \, \chi)} = \kappa_{(f, \, \chi)} \otimes \sigma$ may be taken to be another $\beta_{f}$-extension $\kappa_{(f, \, \chi')}$ as $\sigma$ is a cuspidal representation of $J_{f}/J^{1} \cong k_{L_{f}}^{\times}$ which may be viewed as a quasi-character of $\mathcal{O}_{L_{f}}^{\times}$ that is trivial on $1 + \mathcal{P}_{L_{f}}$. This implies that for a given $\beta_{f}$-extension $\kappa_{(f, \, \chi)}$, the tensor product $\kappa_{(f, \, \chi)} \otimes \sigma$ is another $\beta_{f}$-extension. Moreover, we note that such a simple type is maximal as the period $e\left(\mathfrak{A}_{2N}\right)$ of $\mathfrak{A}_{2N}$ is equal to $e(E_{f}|F) = N$. 

To get an extended maximal simple type $\left(\textbf{J}_{f},  \Lambda_{(f,  \, \chi,  \, \zeta)}\right)$ from a maximal simple type defined above, we set $\Lambda_{(f,  \, \chi,  \, \zeta)} = \zeta \in \mathbb{C}^{\times}$ so that $\Lambda_{(f,  \, \chi,  \, \zeta)}(\varpi_{F}) = \zeta^{-N}\chi\left(\sigma_{f}\right)$. We are now ready to introduce the notion of a \textit{middle supercuspidal representation}.

\begin{definition}\label{MiddleSupercuspidal}
    \normalfont
    We say $\pi_{(f,  \, \chi,  \, \zeta)}$ is a \textit{middle supercuspidal representation} of $\GL(2N,  F)$ if 
    \[
    \pi_{(f,  \, \chi,  \, \zeta)} = c\text{-}\Ind_{\textbf{J}_{f}}^{\GL(2N,  F)}\Lambda_{(f,  \, \chi,  \, \zeta)}
    \]
    where $\left(\textbf{J}_{f},  \Lambda_{(f,  \, \chi,  \, \zeta)}\right)$ is an extended maximal simple type of $\left(J_{f},  \kappa_{(f, \, \chi)}\right)$. Let $\omega_{(f,  \, \chi,  \, \zeta)}$ denote the central character of $\pi_{(f,  \, \chi,  \, \zeta)}$.
\end{definition}

We may further classify middle supercuspidals via the following proposition which gives a parametrization of the set $\mathcal{A}_{\text{middle}}$ of isomorphism classes of middle supercuspidal representations of $\GL(2N,  F)$. 
\begin{proposition}\label{Bijection}
    There exists a bijection between $\mathcal{A}_{\text{middle}}$ and the set of triples $(\bar{f},  \, \chi,  \, \zeta)$ with $\bar{f}$ a monic irreducible degree two polynomial over $k_{F}$, $\chi$ a multiplicative quasi-character of $k_{\bar{f}} = k_{F}[X]/\left\langle \bar{f}(X) \right\rangle$, and $\zeta \in \mathbb{C}^{\times}$.
\end{proposition}
\begin{proof}
    We first prove that if $f_{1}, f_{2} \in \mathcal{O}_{F}[X]$ are two monic degree two polynomials such that $\bar{f}_{i} = f_{i} \mod{\mathcal{P}_{F}}$ is irreducible over $k_{F}$ and $\bar{f} = \bar{f}_{1} = \bar{f}_{2}$, then they give the same middle supercuspidal representations. Since $\beta_{f_{1}}\beta_{f_{2}}^{-1} \in U^{1}$, we have that $J_{f_{1}} = J_{f_{2}}$ and $\textbf{J}_{f_{1}} = \textbf{J}_{f_{2}}$. Furthermore, a direct computation shows that $\psi_{\beta_{f_{1}}} = \psi_{\beta_{f_{2}}}$ as quasi-characters of $U^{1}$ and that $\beta_{f_{1}}\beta_{f_{2}}^{-1} \in \text{Ker } \psi_{\beta_{f_{i}}}$. 
    
    Since $k_{L_{f_{i}}} \cong k_{\bar{f}}$ under the isomorphism $\sigma_{f_{i}} \mapsto X$, we may interpret a quasi-character of $k_{k_{\bar{f}}}^{\times}$ as a quasi-character of $k_{L_{f_{i}}}^{\times}$. This implies $\kappa_{(f_{1}, \, \chi)} = \kappa_{(f_{2}, \, \chi)}$. Furthermore, since $\beta_{f_{1}}\beta_{f_{2}}^{-1} \in \text{Ker } \psi_{\beta_{f_{i}}}$, we have that $\Lambda_{(f_{1}, \, \chi, \, \zeta)} = \Lambda_{(f_{2}, \, \chi, \, \zeta)}$ and $\pi_{(f_{1}, \, \chi, \, \zeta)} = \pi_{(f_{2}, \, \chi, \, \zeta)}$.
    
    The above gives us a well-defined surjective map $(\bar{f},  \, \chi,  \, \zeta) \mapsto \pi_{(f,  \, \chi,  \, \zeta)}$ from the set of triples defined in the hypothesis and $\mathcal{A}_{\text{middle}}$. To prove injectivity, suppose $\pi_{(f_{1},  \, \chi_{1},  \, \zeta_{1})} \cong \pi_{(f_{2},  \, \chi_{2},  \, \zeta_{2})}$. We claim $(\bar{f}_{1},  \, \chi_{1},  \, \zeta_{1}) = (\bar{f}_{2},  \, \chi_{2},  \, \zeta_{2})$.
    
    From \cite[(2.3.2)]{BK}, we have that the simple stratum $\left[\mathfrak{A}_{2N}, 1, 0, \beta_{f_{i}}\right]$ associated with $\pi_{(f_{i},  \, \chi_{i},  \, \zeta_{i})}$ is fundamental. Hence, we may set 
    \[
    Y_{\beta_{f_{i}}} = \beta_{f_{i}}^{N}\varpi_{F} \in A.
    \]
    Reducing the characteristic polynomial $\Phi_{Y_{\beta_{f_{i}}}}(X)$ of $Y_{\beta_{f_{i}}}$ modulo $\mathcal{P}_{F}$, we have 
    \[
    \phi_{\pi_{(f_{i},  \, \chi_{i},  \, \zeta_{i})}}(X) = \bar{f}_{i}.
    \]
    This implies $\bar{f} = \bar{f}_{1} = \bar{f}_{2}$ because the minimal polynomial is an invariant of $\pi_{(f_{i},  \, \chi_{i},  \, \zeta_{i})}$ \cite[Subsection 6.1]{BHK}. Next, we fix a lift $f \in \mathcal{O}_{F}[X]$ of $\bar{f}$.
    
    From \cite[(6.2.4)]{BK}, we have that there exists $g \in \GL(2N,  F)$ such that
    \[
    \Lambda_{(f,  \, \chi_{1},  \, \zeta_{1})}^{g} = \Lambda_{(f,  \, \chi_{2},  \, \zeta_{2})};
    \]
    this implies $\leftidx^{g}\psi_{\beta_{f}} = \psi_{\beta_{f}}$. Hence, \cite[Corollary 3.3.17]{BK} tells us that $g$ is contained inside the normalizer $N_{G}(\psi_{\beta_{f}}) = \textbf{J}_{f}$ of $\psi_{\beta_{f}}$. Therefore, $\Lambda_{(f, \, \chi_{1},  \, \zeta_{1})} = \Lambda_{(f, \, \chi_{2}, \, \zeta_{2})}$ which means $\zeta_{1} = \zeta_{2}$ and $\chi_{1} = \chi_{2}$, thus completing the proof for injectivity.
\end{proof}

\subsection{Explicit Whittaker functions for middle supercuspidal representations}

In this subsection, we compute the explicit Whittaker function $\mathcal{W}_{\left(f, \, \chi, \, \zeta\right)} \in W \left(\pi_{\left(f, \, \chi, \, \zeta\right)}, \psi_{F}\right)$ defined in subsection \ref{ExplicitWhittakerfunctions} of a middle supercuspidal representation $\pi_{\left(f, \, \chi, \, \zeta\right)}$ of $\GL(2N, F)$. 
We first note that 
\[
(J_{f} \cap \N(2N, F))  U^{1} = U^{1}
\]
as $\mathcal{O}_{L_{f}}^{\times} \cap \N(2N, F) = \{I_{2N}\}$. This implies $\Psi_{2N} = \psi_{\beta_{f}}$. 

Next, we compute the associated Bessel function $\mathcal{J}_{\left(f, \, \chi, \, \zeta\right)}$. Let $\mathcal{N} = U^{2} := U^{2}\left(\mathfrak{A}_{2N}\right)$. Then for $g \in \textbf{J}_{f}$, 
\begin{align*}
    \mathcal{J}_{\left(f, \, \chi, \, \zeta\right)}(g) &= \left(U^{1} : U^{2}  \right)^{-1} \displaystyle\sum_{h \in U^{1} \big/U^{2} } \notag \psi_{\beta_{f}}\left(h^{-1}\right)\Lambda_{\left(f, \, \chi, \, \zeta\right)}(gh)\\
    &= \Lambda_{\left(f, \, \chi, \, \zeta\right)}(g). 
\end{align*}

Hence, we may define $\mathcal{W}_{\left(f, \, \chi, \, \zeta\right)} \, \colon \, \GL(2N, F) \to \mathbb{C}$ as 

\hspace{0pt}\resizebox{1.0\linewidth}{!}{
  \begin{minipage}{\linewidth}
\begin{align*}
    g \mapsto  \begin{cases} 
      \hfill \psi_{2N}(u)  \Lambda_{\left(f, \, \chi, \, \zeta\right)}(h) &, \text{ if } g = uh \in \N(2N, F)  \textbf{J}_{f} \text{ with } u \in \N(2N, F), \; h \in \textbf{J}_{f} \\
      \hfill 0 & ,  \text{ otherwise}
   \end{cases}.
\end{align*}
\end{minipage}
}

\subsection{Twisting by quasi-characters of $F^{\times}$}

In this subsection, we choose a finite set $\Xi_{\text{middle}}$ of $q_{F}+1$ quasi-characters of $F^{\times}$ as in subsection \ref{CentralCharacter}. By doing so, we may distinguish between the central characters of two middle supercuspidal representations via their twisted gamma factors with elements of $\Xi_{\text{middle}}$. Next, we compute the twisted gamma factor $\gamma\left(s, \pi_{(f,  \, \chi,  \, \zeta)} \times \lambda, \psi_{F}\right)$ where $\lambda$ is a tamely ramified quasi-character of $F^{\times}$, i.e. having level one. 

For $\pi_{(f,  \, \chi,  \, \zeta)}$ a middle supercuspidal, one may choose any $m_{\pi_{(f,  \, \chi,  \, \zeta)}} > \frac{2}{N}$ as in Proposition \ref{Jiang} since the depth of a middle supercuspidal representation of $\GL(2N, F)$ is $\frac{1}{N}$ \cite{Harris}, \cite[Lemma 2.2]{Xu}. Hence, we choose $m_{\pi_{(f,  \, \chi,  \, \zeta)}} = 2$. Let $\{\mu_{i}\}_{i=1}^{q_{F}-1}$ denote the set of $(q_{F}-1)$ roots of unity in $F$ and consider the finite set $C_{\text{middle}} := \left\{\mu_{i}\varpi_{F}^{-2},  \varpi_{F}^{-3}\right\}_{i=1}^{q_{F}-1}$ of cardinality $q_{F}$. 

For every $\mu_{i}\varpi_{F}^{-2} \in C_{\text{middle}}$, we choose a quasi-character $\chi_{\mu_{i}\varpi_{F}^{-2}}$ of $F^{\times}$  with conductor two such that
\[
\chi_{\mu_{i}\varpi_{F}^{-2}}(1+x) = \psi_{F}(\mu_{i}\varpi_{F}^{-2}x) 
\]
for all $x \in \mathcal{P}_{F}^{2}$. Similarly for $\varpi_{F}^{-3} \in C_{\text{middle}}$, we choose a quasi-character $\chi_{\varpi_{F}^{-3}}$ of $F^{\times}$ with conductor three such that $\chi_{\varpi_{F}^{-3}}(1+x) = \psi_{F}(\varpi_{F}^{-3}x)$ for all $x \in \mathcal{P}_{F}^{2}$. Let $\Xi_{\text{middle}}$ denote the finite set:
\begin{equation}\label{Quasi-Characters}
    \Xi_{\text{middle}} := \left\{ \chi_{\mu_{i}\varpi_{F}^{-2}},  \chi_{\varpi_{F}^{-3}}, \mathbbm{1}_{F}\right\}_{i=1}^{q_{F}-1}.
\end{equation}

Hence, for $\pi$ an irreducible supercuspidal representation of $\GL(2N, F)$ with central character $\omega$, Proposition \ref{PropositionCentral} implies:
\begin{corollary}\label{CorollaryCentral}
    If $\gamma\left(s, \pi_{(f,  \, \chi,  \, \zeta)} \times \tau, \psi_{F}\right) = \gamma\left(s, \pi \times \tau, \psi_{F}\right)$ for all $\tau \in \Xi_{\text{middle}}$, then $\omega_{(f,  \, \chi,  \, \zeta)} = \omega$. 
\end{corollary}

We now compute $\gamma\left(s, \pi_{(f,  \, \chi,  \, \zeta)} \times \lambda, \psi_{F}\right)$ for $\lambda$ a tamely ramified quasi-character of $F^{\times}$. From Theorem \ref{GammaFactor2.7}, we have
\[
\gamma\left(s, \pi_{(f,  \, \chi,  \, \zeta)} \times \lambda, \psi_{F}\right) = \lambda(-1)\frac{\widetilde{\Psi}\left(1-s; \rho(w_{2N, 1})\widetilde{\mathcal{W}}_{(f,  \, \chi,  \, \zeta)},  \lambda^{-1}\right)}{\Psi\left(s; \mathcal{W}_{(f,  \, \chi,  \, \zeta)},  \lambda\right)}.
\]
We first compute $\Psi\left(s; \mathcal{W}_{(f,  \, \chi,  \, \zeta)},  \lambda\right)$. 
\begin{proposition}\label{Proposition1}
   \normalfont 
   $\Psi\left(s; \mathcal{W}_{(f,  \, \chi,  \, \zeta)},  \lambda\right) = \text{vol}\left(1+\mathcal{P}_{F}\right)$.
\end{proposition}
\begin{proof}
    By definition, 
    \begin{align*}
        \Psi\left(s; \mathcal{W}_{(f,  \, \chi,  \, \zeta)},  \lambda\right) &= \int\displaylimits_{F^{\times}}\mathcal{W}_{(f,  \, \chi,  \, \zeta)}\left(\begin{pmatrix}
            h & \\
            & I_{2N-1}
        \end{pmatrix}\right)\lambda(h)|h|^{s-\frac{2N-1}{2}}  dh.
    \end{align*}
    From \cite[Theorem 5.8]{Paskunas}, 
    \[
    \Su\left(\mathcal{W}_{(f,  \, \chi,  \, \zeta)}\right) \cap \text{P}(2N, F) = \N(2N, F)\left(U^{1}\cap \text{P}(2N, F)\right)
    \]
    and 
    \begin{equation*}
        \mathcal{W}_{(f,  \, \chi,  \, \zeta)}(um) = \psi_{2N}(u)\psi_{\beta_{f}}(m)
    \end{equation*}
     for all $u \in \N(2N, F)$ and $m \in U^{1}\cap \text{P}(2N, F)$. This implies $\begin{pmatrix}
            h & \\
            & I_{2N-1}
        \end{pmatrix} = um$ for some $u \in \N(2N, F)$ and $m = (m_{i, j}) \in U^{1}\cap \text{P}(2N, F)$.

    Setting $u^{-1}h = m$, we have $h \in 1 + \mathcal{P}_{F}$ from $m_{1, 1}$. Since $\lambda$ is tamely ramified, our integral simplifies to
    \begin{align*}
        \Psi\left(s; \mathcal{W}_{(f,  \, \chi,  \, \zeta)},  \lambda\right) &= \int\displaylimits_{1+\mathcal{P}_{F}}\psi_{\beta_{f}}\left(\begin{pmatrix}
            h & \\
            & I_{2N-1}
        \end{pmatrix}\right) dh \\
        &= \text{vol}\left(1+\mathcal{P}_{F}\right).
    \qedhere \end{align*}
\end{proof}

We now calculate $\widetilde{\Psi}\left(1-s; \rho(w_{2N, 1})\widetilde{\mathcal{W}}_{(f,  \, \chi,  \, \zeta)},  \lambda^{-1}\right)$. Let $e_{i}$ denote the standard column matrix with a one in the $i^{\text{th}}$ row and zeros elsewhere. 
\begin{proposition}\label{Proposition2}
   \normalfont 
   \begin{align*}
       &\widetilde{\Psi}\left(1-s; \rho(w_{2N, 1})\widetilde{\mathcal{W}}_{(f,  \, \chi,  \, \zeta)},  \lambda^{-1}\right) \\
       &=\zeta^{-1}  \lambda\left(-c^{-1}\varpi_{F}^{2}\right)   q_{F}^{1-2s} \, \text{vol}\left(1+\mathcal{P}_{F}\right).
   \end{align*}
\end{proposition}
\begin{proof}
    Unfolding the definition of $\widetilde{\Psi}\left(1-s; \rho(w_{2N, 1})\widetilde{\mathcal{W}}_{(f,  \, \chi,  \, \zeta)},  \lambda^{-1}\right)$, we have 
\begin{align*}
    & \widetilde{\Psi}\left(1-s; \rho(w_{2N, 1})\widetilde{\mathcal{W}}_{(f,  \, \chi,  \, \zeta)},  \lambda^{-1}\right) = \\
    & \int\displaylimits_{F^{\times}}\int\displaylimits_{\Mat(2 \times 1, F)}\mathcal{W}_{(f,  \, \chi,  \, \zeta)}\left(\begin{pmatrix}
         & 1 &  &  &\\
         &  & 1 &  &\\
         &  &  & \ddots & \\
         & & & & 1 \\
        h^{-1} &  & -x_{2N-2}h^{-1} & \cdots & -x_{1}h^{-1}
    \end{pmatrix}\right)\lambda(h)^{-1}|h|^{\frac{3-2N}{2}-s}  dx  dh.
\end{align*}
\noindent Next, we check to see when an element $\alpha$ of the form 
\begin{equation}\label{GL(1)}
   \alpha = \begin{pmatrix}
         & 1 &  &  &\\
         &  & 1 &  &\\
         &  &  & \ddots & \\
         & & & & 1 \\
        h^{-1} &  & -x_{2N-2}h^{-1} & \cdots & -x_{1}h^{-1} 
    \end{pmatrix}
\end{equation}
with $h \in F^{\times}$ and $x_{i} \in F$ lies inside the support $\Su\left(\mathcal{W}_{(f,  \, \chi,  \, \zeta)}\right)$ of $\mathcal{W}_{(f,  \, \chi,  \, \zeta)}$ where 
\begin{equation}\label{Support}
\Su\left(\mathcal{W}_{(f,  \, \chi,  \, \zeta)}\right) \subset \displaystyle\bigsqcup_{k \in \mathbb{Z}}\N(2N, F)  \beta_{f}^{k}  J_{f}
\end{equation}
\cite[Proposition 5.7]{Paskunas}. 

Before introducing the following lemma, we make several observations about powers of $\beta_{f}$ that will aid us in the proof:
\begin{enumerate}
    \item For any $k \in \mathbb{Z}$, $\beta_{f}^{k}$ has at most two non-zero entries in each row and column. Furthermore, given a row $r$ of $\beta_{f}^{k}$ with two non-zero entries $(\beta_{f}^{k})_{r,  i}$ and $(\beta_{f}^{k})_{r,  j}$, $|i-j|=N$ where $|\cdot|$ denotes the standard Archimedean absolute value.
    \item Given a row $r$ of $\beta_{f}^{k}$, there exists a column $j$ of $\beta_{f}^{k}$ such that at least two of the entries $(\beta_{f}^{k})_{r,  j}$, $(\beta_{f}^{k})_{r,  j+N}$, $(\beta_{f}^{k})_{r+N,  j}$, and $(\beta_{f}^{k})_{r+N,  j+N}$ are non-zero. Likewise, given a column $j$ of $\beta_{f}^{k}$, there exists a row $r$ of $\beta_{f}^{k}$ that satisfies this property. Furthermore, 
    \begin{equation*}
        \begin{pmatrix}
            (\beta_{f}^{k})_{r,  j} &(\beta_{f}^{k})_{r,  j+N} \\
            (\beta_{f}^{k})_{r+N,  j} & (\beta_{f}^{k})_{r+N,  j+N}
        \end{pmatrix} = \begin{pmatrix}
            A_{3}\varpi_{F}^{n} & cA_{1}\varpi_{F}^{n-1} \\
            A_{1}\varpi_{F}^{n+1} & A_{2}\varpi_{F}^{n}
        \end{pmatrix}
    \end{equation*}
        \noindent where $n \in \mathbb{Z}$ and $A_{i} \in \mathcal{O}_{F}$ with $cA_{1}^{2}-A_{2}A_{3} = c^{m}$ for some $m \geq 0$.  
\end{enumerate}

\begin{lemma}\label{Lemma1}
    $\alpha \in \Su\left(\mathcal{W}_{(f,  \, \chi,  \, \zeta)}\right)$ if and only if  $k = -1$. 
\end{lemma}
\begin{proof}

    Let $\alpha$ be an element of the form given in (\ref{GL(1)}) such that $\alpha \in \Su\left(\mathcal{W}_{(f,  \, \chi,  \, \zeta)}\right)$. Then $\alpha = (\alpha_{i,  j}) = u \beta_{f}^{k}  x$ for some $u = (u_{i,  j}) \in \N(2N, F)$, $x = (x_{i,  j}) \in J_{f}$, and $k \in \mathbb{Z}$.


    Suppose $k < -1$. Then $u^{-1} \alpha = \beta_{f}^{k} x$ which forces $\beta_{f}^{k} x$ to have second column equal to $e_{1}$. From (\ref{Order}), we see that $x_{i, 2} \in \mathcal{O}_{F}$ for $1 \leq i \leq 2$ and $x_{i, 2} \in \mathcal{P}_{F}$ for $2 < i \leq 2N$. Using the first observation, we have that the first row of $\beta_{f}^{k}$ has at most two non-zero entries $(\beta_{f}^{k})_{1,  j}$ and $(\beta_{f}^{k})_{1,  j+N}$ where $(\beta_{f}^{k})_{1,   j} \in \mathcal{P}_{F}$ if $j \leq 2$ and $(\beta_{f}^{k})_{1,   j} \in \mathcal{O}_{F}$ if $j > 2$. Comparing $(u^{-1} \alpha)_{1, 2}$ with $\left(\beta_{f}^{k} x\right)_{1, 2} \in \mathcal{P}_{F}$, we have a contradiction as $(u^{-1} \alpha)_{1, 2} = 1$ and $k \geq -1$.
 
    Now, suppose $k \geq -1$. From (\ref{Order}), we set $x_{2, 2} = y_{2, 2}$ and $x_{N+2, 2} = y_{N+2, 2}\varpi_{F}$ with $y_{2, 2}$ and $y_{N+2, 2} \in \mathcal{O}_{F}$. We note that if $y_{2, 2}$ and $y_{N+2, 2} \in \mathcal{P}_{F}$, then $\det(x) \in \mathcal{P}_{F}$. To see this, when we conjugate an element $x \in \mathfrak{A}_{2N}$ with $y_{2, 2}$ and $y_{N+2, 2} \in \mathcal{P}_{F}$ by $g_{f}^{-1}$ to get $x^{g_{f}^{-1}} \in \mathfrak{A}_{2}$, we have that the second $2 \times 2$ block on the diagonal satisfies
    \[
    \begin{pmatrix}
        x^{g_{f}^{-1}}_{3, 3} & x^{g_{f}^{-1}}_{3, 4} \\
        x^{g_{f}^{-1}}_{4, 3} & x^{g_{f}^{-1}}_{4, 4}
    \end{pmatrix}  \in \begin{pmatrix}
        \mathcal{P}_{F} & \mathcal{O}_{F} \\
        \mathcal{P}_{F} & \mathcal{O}_{F}
    \end{pmatrix};
    \]
    this implies $\det(x) \in \mathcal{P}_{F}$. 
    
    Let $r$ be a row of $\beta_{f}^{k}$ such that at least two of the entries $(\beta_{f}^{k})_{r,  2}$, $(\beta_{f}^{k})_{r,  N+2}$, $(\beta_{f}^{k})_{r+N,  2}$, and $(\beta_{f}^{k})_{r+N,  N+2}$ are non-zero. From the second observation, 
    \[
    \begin{pmatrix}
            (\beta_{f}^{k})_{r,  2} &(\beta_{f}^{k})_{r,  N+2} \\
            (\beta_{f}^{k})_{N+r,  2} & (\beta_{f}^{k})_{N+r,  N+2}
        \end{pmatrix} = \begin{pmatrix}
            A_{3}\varpi_{F}^{n} & cA_{1}\varpi_{F}^{n-1} \\
            A_{1}\varpi_{F}^{n+1} & A_{2}\varpi_{F}^{n}
        \end{pmatrix}
    \]
    where $A_{i} \in \mathcal{O}_{F}$, $n \leq 0$, and $cA_{1}^{2}-A_{2}A_{3} = c^{m} \in \mathcal{O}_{F}^{\times}$ for some $m \geq 0$. Since $u^{-1} \alpha = \beta_{f}^{k}x$, the second column of $\beta_{f}^{k}x$ is equal to $e_{1}$. This fact along with (\ref{Order}) give us the system of equations:
        \begin{align*}
            A_{3}y_{2, 2}+cA_{1}y_{N+2,  2} &= \delta(r)\varpi_{F}^{-n} \\
            A_{1}y_{2, 2}+A_{2}y_{N+2,  2} &= 0 
        \end{align*}
    where $n \leq 0$ with $\delta(r) = 1$ if $r = 1$ and zero otherwise. 

    Solving the system of equations above, we have 
    \[
    \left(y_{2, 2},  y_{N+2,  2}\right) = \left(-\frac{\delta(r)A_{2}\varpi_{F}^{-n}}{c^{m}},  \frac{\delta(r)A_{1}\varpi_{F}^{-n}}{c^{m}}\right).
    \]
    This implies that an element of the form given in $(\ref{GL(1)})$ is contained inside the support of $\mathcal{W}_{(f,  \, \chi,  \, \zeta)}$ if and only if $r = 1$ and $n = 0$ if and only if $k=-1$.
\end{proof}

From Lemma \ref{Lemma1}, we have that $\alpha = u \beta_{f}^{-1} x$ for some $u \in \N(2N, F)$ and $x \in J_{f}$. Since $u^{-1}\alpha = \beta_{f}^{-1} x$, the second column of $\beta_{f}^{-1} x$ is equal to $e_{1}$. This tells us that the second column of $x$ is equal to $e_{2}$. Writing $x = az'$ for some $a = a_{0}c + a_{1}\sigma_{f} \in \mathcal{O}_{L_{f}}^{\times}$ and $z' \in U^{1}$ as in (\ref{decomposition}), we have $a_{0}c = 1$ as $x_{2, 2} \in a_{0}c + \mathcal{P}_{F}$. 

When we look at $\left(\beta_{f}^{-1} x\right)_{2N,  1}$, (\ref{decomposition}) also tells us that $h^{-1} \in c^{-1}\varpi_{F}^{2}  (1+ \mathcal{P}_{F})$ as $x_{1, 1} \in a_{0}c + \mathcal{P}_{F}$ and $(u^{-1}\alpha)_{2N,  1} = h^{-1}$. Furthermore, $a_{1}\in \mathcal{P}_{F}$ as $x_{N+2,  2}=0$ since the second column of $\beta_{f}^{-1} x$ is equal to $e_{1}$. Hence, 
\[
-x_{i}h^{-1} \in \begin{cases}
    \mathcal{P}_{F}^{2}, \text{ if }  N \leq i \leq 2N-2 \\
    \mathcal{P}_{F}, \text{ otherwise} 
\end{cases} \text{which implies \;} x_{i} \in \begin{cases}
     \mathcal{O}_{F}, \text{ if }  N \leq i \leq 2N-2 \\
     \mathcal{P}_{F}^{-1}, \text{ otherwise} 
\end{cases}.
\]
With these conditions, we may factorize $\alpha$ as 
\begin{align*}
\alpha = \begin{pmatrix}
    1 & & & & \\
    & \ddots & & & \\
    & & 1 & & d\varpi_{F}^{-1} \\
    & & & \ddots & \\
    & & & & 1
\end{pmatrix}\beta_{f}^{-1}
z
\end{align*}
where $d\varpi_{F}^{-1} = u_{N, 2N}$ and 
\[
z = \begin{pmatrix}
    c\varpi_{F}^{-2}h^{-1} & & -x_{2N-2}h^{-1}c\varpi_{F}^{-2} & -x_{2N-1}h^{-1}c\varpi_{F}^{-2} & \cdots & -x_{1}h^{-1}c\varpi_{F}^{-2} \\
    & 1 & & & & \\
    & & 1 & & & \\
    & & & \ddots & & \\
    & & & & & \\
    & & & & & \\
    & & & & & 1
\end{pmatrix} \in U^{1}.
\]

Rewriting $\widetilde{\Psi}$, we see that
\begin{align*}
      \widetilde{\Psi}&\left(1-s; \rho(w_{2N, 1}) \widetilde{\mathcal{W}}_{(f,  \, \chi,  \, \zeta)},  \lambda^{-1}\right)  
      = \int\displaylimits_{c\varpi_{F}^{-2}(1+ \mathcal{P}_{F})}\int\displaylimits_{\mathcal{O}_{F}}\cdots\int\displaylimits_{\mathcal{O}_{F}}\int\displaylimits_{\mathcal{P}_{F}^{-1}}\cdots \int\displaylimits_{\mathcal{P}_{F}^{-1}}dx_{1}\cdots dx_{2N-2}   \\
     &  \cdot \mathcal{W}_{(f,  \, \chi,  \, \zeta)}\left(\begin{pmatrix}
         & 1 &  &  &\\
         &  & 1 &  &\\
         &  &  & \ddots & \\
         & & & & 1 \\
        h^{-1} &  & -x_{2N-2}h^{-1} & \cdots & -x_{1}h^{-1}
    \end{pmatrix}\right)\lambda(h)^{-1}|h|^{\frac{3-2N}{2}-s}  dh.
\end{align*}
We note that the $x_{i}$ terms do not affect the value of $\mathcal{W}_{(f,  \, \chi,  \, \zeta)}$. Furthermore, since we have normalized our Haar measure to have volume $q_{F}^{1/2}$ on $\mathcal{O}_{F}$, the volume of $\mathcal{P}_{F}^{-1}$ is equal to $q_{F}^{3/2}$. Hence, we have that the above is equal to 
\begin{align*}
    & \zeta^{-1}q_{F}^{2N-2}\int\displaylimits_{\frac{c}{\varpi_{F}^{2}}(1+ \mathcal{P}_{F})}\lambda(h)^{-1}|h|^{\frac{3-2N}{2}-s}  dh \\
    &= \zeta^{-1}\lambda(c^{-1}\varpi_{F}^{2}) \, q_{F}^{1-2s} \, \text{vol}(1+\mathcal{P}_{F})
\end{align*}
from the left translation invariance of our Haar measure $dh$ and the fact that $\lambda$ is tamely ramified. 
\end{proof}

Propositions \ref{Proposition1} and \ref{Proposition2} now give us the following formula for $\gamma\left(s, \pi_{(f,  \, \chi,  \, \zeta)} \times \lambda, \psi_{F}\right)$:
\begin{corollary}\label{Corollary1}
    $\gamma\left(s, \pi_{(f,  \, \chi,  \, \zeta)} \times \lambda, \psi_{F}\right) = \zeta^{-1}\lambda(-c^{-1}\varpi_{F}^{2}) \, q_{F}^{1-2s}$.
\end{corollary}

From Corollary \ref{Corollary1}, we may deduce two key properties that two middle supercuspidals share if their twisted gamma factors by tamely ramified quasi-characters are equal. For the following proposition, let $f_{i} = X^{2}-d_{i}X-c_{i}$.
\begin{proposition}\label{Proposition3}
    Suppose  $\gamma(s, \pi_{(f_{1}, \chi_{1}, \zeta_{1})} \times \lambda, \psi_{F}) = \gamma(s, \pi_{(f_{2}, \chi_{2}, \zeta_{2})} \times \lambda, \psi_{F})$ for all tamely ramified quasi-characters $\lambda$ of $F^{\times}$. Then $\zeta_{1} = \zeta_{2}$ and $c_{1} \equiv c_{2} \mod{\mathcal{P}_{F}}$.
 
\end{proposition}
\begin{proof}
    Suppose 
    \[
    \gamma(s, \pi_{(f_{1}, \chi_{1}, \zeta_{1})} \times \lambda, \psi_{F}) = \gamma(s, \pi_{(f_{2}, \chi_{2}, \zeta_{2})} \times \lambda, \psi_{F})
    \]
    for all tamely ramified quasi-characters $\lambda$ of $F^{\times}$. Then Corollary \ref{Corollary1} tells us that
    \[
    \zeta_{1}^{-1}\lambda(-c_{1}^{-1}\varpi_{F}^{2}) q_{F}^{1-2s} \\ 
        =\zeta_{2}^{-1}\lambda(-c_{2}^{-1}\varpi_{F}^{2}) q_{F}^{1-2s}.
    \]
    Setting $\lambda$ as the trivial character (standard gamma factor), we have that 
    \[
    \zeta_{1} = \zeta_{2}.
    \]
    
    Therefore,  
    \[
    \lambda(-c_{1}^{-1}\varpi_{F}^{2}) = \lambda(-c_{2}^{-1}\varpi_{F}^{2})
    \]
    for all tamely ramified quasi-characters $\lambda$ which implies $c_{1} \equiv c_{2} \mod{\mathcal{P}_{F}}$.
\end{proof}

\subsection{Twisting by simple supercuspidal representations of $\GL(N, F)$}

In this subsection, we twist a middle supercuspidal representation $\pi_{(f,  \, \chi,  \, \zeta)}$ of $\GL(2N, F)$ by a simple supercuspidal representation $\pi_{\left(u,  \, \phi,  \, \zeta'\right)}$ of $\GL(N, F)$ where $u \in \mathcal{O}_{F}^{\times}$. Furthermore, we use this twisting to distinguish middle supercuspidals from each other. 

Let 
\[
g_{u} := \begin{pmatrix}
    1 & & (u\varpi_{F})^{-1} & &  \\
    & \ddots & & &  \\
    & & 1 & &  \\
    & & & \ddots & \\
    & & & & 1
\end{pmatrix}
\]
where $(u\varpi_{F})^{-1} = (g_{u})_{1,  N+1}$ and $R(g_{u})  \mathcal{W}_{(f,  \, \chi,  \, \zeta)} \in W\left(\pi_{(f,  \, \chi,  \, \zeta)}, \psi_{F}\right)$ be the Whittaker function that right translates $\mathcal{W}_{(f,  \, \chi,  \, \zeta)}$ by $g_{u}$. Then Theorem \ref{GammaFactor2.7} gives us the equation
\begin{equation}\label{GammaSimple}
    \gamma\left(s, \pi_{(f,  \, \chi,  \, \zeta)} \times \pi_{\left(u,  \, \phi,  \, \zeta'\right)}, \psi_{F}\right)
    = \omega_{\left(u,  \, \phi,  \, \zeta'\right)}(-1)\frac{\widetilde{\Psi}\left(1-s; \rho(w_{2N, N})\widetilde{R(g_{u})  \mathcal{W}}_{(f,  \, \chi,  \, \zeta)}, \widetilde{\mathcal{W}}_{\left(u,  \, \phi,  \, \zeta'\right)}\right)}{\Psi\left(s; R(g_{u})  \mathcal{W}_{(f,  \, \chi,  \, \zeta)}, \mathcal{W}_{\left(u,  \, \phi,  \, \zeta'\right)}\right)}. 
\end{equation}

We first calculate $\Psi\left(s; R(g_{u})  \mathcal{W}_{(f,  \, \chi,  \, \zeta)}, \mathcal{W}_{\left(u,  \, \phi,  \, \zeta'\right)}\right)$. 
\begin{proposition}\label{Proposition4}
\normalfont
    $\Psi\left(s; R(g_{u})  \mathcal{W}_{(f,  \, \chi,  \, \zeta)}, \mathcal{W}_{\left(u,  \, \phi,  \, \zeta'\right)}\right) = \text{vol}\left(U^1\left(\mathfrak{I}_{N}\right)\right)$.
\end{proposition}
\begin{proof}
    By definition,
    \begin{align*}
       & \Psi\left(s; R(g_{u})  \mathcal{W}_{(f,  \, \chi,  \, \zeta)}, \mathcal{W}_{\left(u,  \, \phi,  \, \zeta'\right)}\right) \\
       & = \int\displaylimits_{\N(N, F) \backslash \GL(N, F)} R(g_{u})  \mathcal{W}_{(f,  \, \chi,  \, \zeta)}\left(\begin{pmatrix}
        h &  \\
         & I_{N} 
    \end{pmatrix}\right)\mathcal{W}_{\left(u,  \, \phi,  \, \zeta'\right)}(h)|\det(h)|^{s-\frac{N}{2}}  dh
    \end{align*}
    where $h = (h_{i, j})$.
    From \cite[Theorem 5.8]{Paskunas}, 
    \[
    \Su\left(\mathcal{W}_{(f,  \, \chi,  \, \zeta)}\right) \cap \text{P}(2N,  F) = \N(2N,  F)\left(U^{1} \cap \text{P}(2N,  F)\right)
    \] 
    and 
    \begin{equation*}
        \mathcal{W}_{(f,  \, \chi,  \, \zeta)}(nm) = \psi_{2N}(n)\psi_{\beta_{f}}(m)
    \end{equation*}
    for all $n \in \N(2N,  F)$ and $m \in U^{1}\cap \text{P}(2N,  F)$. This implies 
    $\begin{pmatrix}
        h &  \\
         & I_{N} 
    \end{pmatrix}g_{u} = nm$
    for some $n = (n_{i, j}) \in \N(2N,  F)$ and $m \in U^{1}\cap \text{P}(2N,  F)$. Since $h \in \N(N, F) \backslash \GL(N, F)$, we may choose $n$ such that $n_{i, j} = 0$ for $1 \leq i < j \leq N$.
    
    Rewriting the equation above as $n^{-1}\begin{pmatrix}
        h &  \\
         & I_{N} 
    \end{pmatrix}g_{u} = m$, we have $h \in U^{1}\left(\mathfrak{I}_{N}\right)$ from $(\ref{Jacobson})$ and the fact that the upper left $N \times N$ block of the left hand side is equal to $h$. We may factorize $\begin{pmatrix}
        h &  \\
         & I_{N} 
    \end{pmatrix}g_{u}$ by setting $n = (n_{i, j})$ as: $n_{i, N+1} = \frac{h_{i, 1}}{u\varpi_{F}}$ for $1 \leq i \leq N$ and zeros elsewhere above the diagonal, and $m = \begin{pmatrix}
        h &  \\
         & I_{N} 
    \end{pmatrix}$. This implies
\[
R(g_{u})  \mathcal{W}_{(f,  \, \chi,  \, \zeta)}\left(\begin{pmatrix}
        h &  \\
         & I_{N} 
    \end{pmatrix}\right) = \psi_{F}\left(\displaystyle\sum_{i=1}^{N-1}h_{i, i+1} + \frac{h_{N, 1}}{u\varpi_{F}}\right)
\]
and
\[
\mathcal{W}_{\left(u,  \, \phi,  \, \zeta'\right)}(h) = \psi_{F}^{-1}\left(\displaystyle\sum_{i=1}^{N-1}h_{i, i+1} + \frac{h_{N, 1}}{u\varpi_{F}}\right).
\]
Hence, 
    \[
    \Psi\left(s; R(g_{u})  \mathcal{W}_{(f,  \, \chi,  \, \zeta)}, \mathcal{W}_{\left(u,  \, \phi,  \, \zeta'\right)}\right) =\text{vol}\left(U^1\left(\mathfrak{I}_{N}\right)\right). \qedhere
    \]
\end{proof}

Now, we show that twisting by simple supercuspidals of $\GL(N, F)$ enables us to distinguish middle supercuspidal representations from each other.
\begin{proposition}\label{Proposition5}
Let $a = \frac{u}{cu^{2}+du-1}$. Then 
    \[
    \gamma\left(s, \pi_{(f,  \, \chi,  \, \zeta)} \times \pi_{\left(u,  \, \phi,  \, \zeta'\right)}, \psi_{F}\right) 
         = \zeta^{-N}\chi(acu + a\sigma_{f}) \, \omega_{\left(u,  \, \phi,  \, \zeta'\right)}(a u\varpi_{F}^{2})  P\left(q_{F}^{-s}\right)
    \]
    where $P \in \mathbb{C}\left(q_{F}^{-s}\right)$ is a non-zero rational function that does not depend on $\omega_{\left(u,  \, \phi,  \, \zeta'\right)}$, $\chi$, or $\zeta$.
\end{proposition}
\begin{proof}
    We first compute $\widetilde{\Psi}\left(1-s; \rho(w_{2N, N})\widetilde{R(g_{u})  \mathcal{W}}_{(f,  \, \chi,  \, \zeta)}, \widetilde{\mathcal{W}}_{\left(u,  \, \phi,  \, \zeta'\right)}\right)$. Using a change of variables for $\text{Re} (s) \ll 0$ \cite[Proof of Theorem 3.1]{Ye2}, we have 

\begin{align*}
    & \widetilde{\Psi}\left(1-s; \rho(w_{2N, N})\widetilde{R(g_{u})  \mathcal{W}}_{(f,  \, \chi,  \, \zeta)}, \widetilde{\mathcal{W}}_{\left(u,  \, \phi,  \, \zeta'\right)}\right)  \\
    &= \!\!\!\!\int\displaylimits_{\N(N, F) \backslash \GL(N, F)}\,\int\displaylimits_{\Mat(N \times N-1, F)} \!\!\!\!\!\!\!\!\!\!\! R(g_{u})  \mathcal{W}_{(f,  \, \chi,  \, \zeta)}\left(\begin{pmatrix}
          & 1 & \\
          &   & I_{N-1} \\
          h & & x
    \end{pmatrix}\right)\mathcal{W}_{\left(u,  \, \phi,  \, \zeta'\right)}(h)|\det(h)|^{s-2+\frac{N}{2}}  dx  dh.
\end{align*}

Next, we check to see when an element $\alpha$ of the form 
\begin{equation}\label{GL(2)}
    \alpha = \begin{pmatrix}
          & 1 & \\
          &   & I_{N-1} \\
          h & & x
    \end{pmatrix}g_{u}
\end{equation}
with $h = (h_{i, j}) \in \N(N, F) \backslash \GL(N, F)$ and $x \in \Mat(N \times N-1, F)$ is contained inside (\ref{Support}).

\begin{lemma}\label{Lemma2}
    $\alpha \in \Su\left(\mathcal{W}_{(f,  \, \chi,  \, \zeta)}\right)$ if and only if $k = -N$. 
\end{lemma}
\begin{proof}
    Suppose an element $\alpha$ of the form in (\ref{GL(2)}) lies inside $\Su\left(\mathcal{W}_{(f,  \, \chi,  \, \zeta)}\right)$. Then $\alpha = (\alpha_{i, j}) =  n\beta_{f}^{k}x$ where $n = (n_{i, j}) \in \N(2N,  F)$, $x = (x_{i, j})  \in J_{f}$, and $k \in \mathbb{Z}$. Next, we set $x_{1, 1} = y_{1, 1}$, $x_{N+1,  1} = y_{N+1,  1} \varpi_{F}$, $x_{1,  N+1} = y_{1,  N+1} \varpi_{F}^{-1}$, and $x_{N+1,  N+1} = y_{N+1,  N+1} $ where $y_{i, j} \in \mathcal{O}_{F}$ from (\ref{Order}). These terms will aid in the proof of this lemma because if $y_{1, 1} \equiv u y_{1,  N+1} \mod{\mathcal{P}_{F}}$ and $y_{N+1,  1} \equiv u y_{N+1,  N+1} \mod{\mathcal{P}_{F}}$, then $\det(x) \in \mathcal{P}_{F}$. To see this, when we conjugate $x \in \mathfrak{A}_{2N}$ with $y_{1, 1} \equiv u y_{1,  N+1} \mod{\mathcal{P}_{F}}$ and $y_{N+1,  1} \equiv u y_{N+1,  N+1} \mod{\mathcal{P}_{F}}$ by $g_{f}^{-1}$ to get $x^{g_{f}^{-1}} \in \mathfrak{A}_{2}$, we have that the first $2 \times 2$ block on the diagonal satisfies
    \[
    \begin{pmatrix}
        x^{g_{f}^{-1}}_{1, 1} & x^{g_{f}^{-1}}_{1, 2} \\
        x^{g_{f}^{-1}}_{2, 1} & x^{g_{f}^{-1}}_{2, 2}
    \end{pmatrix}  = \begin{pmatrix}
        uy_{1,  N+1} & c^{-1}y_{1,  N+1} \\
        cuy_{N+1,  N+1} & y_{N+1,  N+1}
    \end{pmatrix};
    \]
    this implies $\det(x) \in \mathcal{P}_{F}$.

    
     Suppose $k<-N$, then $n^{-1}\alpha = \beta_{f}^{k}x$. When we compare $(n^{-1}\alpha)_{1 , 1}$ and $(n^{-1}\alpha)_{1 , N+1}$ with $\left(\beta_{f}^{k}x\right)_{1,  1}$ and $\left(\beta_{f}^{k}x\right)_{1,  N+1}$ respectively, we have $\left(\beta_{f}^{k}x\right)_{1,  N+1} = (u\varpi_{F})^{-1}\left(\beta_{f}^{k}x\right)_{1, 1} + 1$. When $k < -N$, we have that $\left(\beta_{f}^{k}x\right)_{1,  N+1}$ and $(u\varpi_{F})^{-1}\left(\beta_{f}^{k}x\right)_{1, 1} \in \mathcal{P}_{F}$ which implies $k \geq -N$. 
    
    Now suppose $k \geq -N$. Then $n^{-1}\alpha = \beta_{f}^{k}x$. Let us recall the observations given before the proof of Lemma \ref{Lemma1}. From the second observation, let $r$ be the row of $\beta_{f}^{k}$ such that at least two of the entries $(\beta_{f}^{k})_{r,  1}$, $(\beta_{f}^{k})_{r,  N+1}$, $(\beta_{f}^{k})_{N+r,  1}$, and $(\beta_{f}^{k})_{N+r,  N+1}$ are non-zero with 
    \[
    \begin{pmatrix}
            (\beta_{f}^{k})_{r,  1} &(\beta_{f}^{k})_{r,  N+1} \\
            (\beta_{f}^{k})_{N+r,  1} & (\beta_{f}^{k})_{N+r,  N+1}
        \end{pmatrix} = \begin{pmatrix}
            A_{3}\varpi_{F}^{m} & cA_{1}\varpi_{F}^{m-1} \\
            A_{1}\varpi_{F}^{m+1} & A_{2}\varpi_{F}^{m}
        \end{pmatrix}
    \]
    where $m \leq 1$, $A_{i} \in \mathcal{O}_{F}$, and $cA_{1}^{2}-A_{2}A_{3} \in \mathcal{O}_{F}^{\times}$.

    From (\ref{Order}), we compare
    \[
    \begin{pmatrix}
        (n^{-1}\alpha)_{r, 1} & (n^{-1}\alpha)_{r,  N+1} \\ 
        (n^{-1}\alpha)_{N+r,  1} & (n^{-1}\alpha)_{N+r,  N+1}
    \end{pmatrix} = \begin{pmatrix}
        (\beta_{f}^{k}x)_{r, 1} & (\beta_{f}^{k}x)_{r,  N+1} \\ 
        (\beta_{f}^{k}x)_{N+r,  1} & (\beta_{f}^{k}x)_{N+r,  N+1}
    \end{pmatrix}
    \]
    and receive the following system of equations:
         \begin{align}\label{Linear}
             A_{3}  y_{1, 1} + cA_{1}  y_{N+1,  1} &= u
             \left(A_{3}  y_{1,  N+1} + cA_{1}  y_{N+1,  N+1} + \delta\left(r\right)\varpi_{F}^{1-m}\right) \\ \notag
             A_{1}  y_{1, 1} + A_{2}  y_{N+1,  1} &= u
            \left(A_{1}  y_{1,  N+1} + A_{2}  y_{N+1,  N+1}\right) 
         \end{align}
    where $\delta(r) = 1$ if either $\left(\beta_{f}^{k}\right)_{1, 1}$ or $\left(\beta_{f}^{k}\right)_{1,  N+1}$ is non-zero and zero otherwise. We can rewrite our system of equations as
    \[
    \begin{pmatrix}
            A_{3} & cA_{1} \\
            A_{1} & A_{2}
        \end{pmatrix}\begin{pmatrix}
            y_{1,  1} - uy_{1,  N+1} \\
            y_{N+1,  1} - uy_{N+1,  N+1}
        \end{pmatrix} = \begin{pmatrix}
            \delta\left(r\right)\varpi_{F}^{1-m} \\
            0
        \end{pmatrix}.
    \]
    Hence, an element of the form in (\ref{GL(2)}) is contained inside the support of $\mathcal{W}_{(f,  \, \chi,  \, \zeta)}$ if and only if $r = 1$ and $m=1$ if and only if $k = -N$.   
\end{proof}

From Lemma \ref{Lemma2}, we have that $n^{-1}\alpha =  \beta_{f}^{-N}x$ for some $n \in \N(2N,  F)$ and $x \in J_{f}$. Since $h \in \N(N, F) \backslash \GL(N, F)$, we may choose $n$ such that $n_{i, j} = 0$ for $N < i < j \leq 2N$. Looking at row $N+1$ of $n^{-1}\alpha$, we see that $\beta_{f}^{-N}x$ must satisfy $y_{1,  1} =  u  y_{1,  N+1} $. When we write $x = az$ for some $a = a_{0}c + a_{1}\sigma_{f} \in \mathcal{O}_{L_{f}}^{\times}$ and $z \in U^{1}$ as in (\ref{decomposition}), we see that $y_{1, 1} \in a_{0}c + \mathcal{P}_{F}$ and $y_{1,  N+1} \in a_{1}c + \mathcal{P}_{F}$. Hence, 
\begin{equation}\label{FirstCongruence}
    a_{0} \equiv a_{1}u \mod{\mathcal{P}_{F}}.
\end{equation}

Next, we solve the corresponding system of linear equations in (\ref{Linear}) and see that $y_{N+1,  1} = u(y_{N+1, N+1}-1)$ where $y_{N+1,  1} \in a_{1} + \mathcal{P}_{F}$ and $y_{N+1,  N+1} \in (a_{0}c + a_{1}d) + \mathcal{P}_{F}$ from (\ref{decomposition}). This equality along with (\ref{FirstCongruence}) imply
\[
    a_{1} \equiv a:= \frac{u}{cu^{2}+du-1} \mod{\mathcal{P}_{F}}.
\]
Moreover, since the bottom left $N \times N$ block of $n^{-1}\alpha$ is equal to $h$, (\ref{decomposition}) also tells us that $h \in au\varpi_{F}^{2}  U^{1}\left(\mathfrak{I}_{N}\right)$. 
 
The above implies that we may factor from $\widetilde{\Psi}\left(1-s; \rho(w_{2N, N})\widetilde{R(g_{u})  \mathcal{W}}_{(f,  \, \chi,  \, \zeta)}, \widetilde{\mathcal{W}}_{\left(u,  \, \phi,  \, \zeta'\right)}\right)$ the terms $\zeta^{-N}$, $\chi\left(acu+a\sigma_{f}\right)$, and $\omega_{\left(u,  \, \phi,  \, \zeta'\right)}\left(au\varpi_{F}^{2}\right)$. We then see that the remaining integral is a non-zero rational function $Q\left(q_{F}^{-s}\right) \in \mathbb{C}\left(q_{F}^{-s}\right)$ that does not depend on $\omega_{\left(u,  \, \phi,  \, \zeta'\right)}$, $\chi$, or $\zeta$. From (\ref{GammaSimple}), 
\begin{align*}
    &\gamma\left(s, \pi_{(f,  \, \chi,  \, \zeta)} \times \pi_{\left(u,  \, \phi,  \, \zeta'\right)}, \psi_{F}\right) \\
    &= \zeta^{-N}\chi(acu + a\sigma_{f}) \, \omega_{\left(u,  \, \phi,  \, \zeta'\right)}(au\varpi_{F}^{2})  P\left(q_{F}^{-s}\right)
\end{align*}
where $P\left(q_{F}^{-s}\right) \in \mathbb{C}\left(q_{F}^{-s}\right)$ is a non-zero rational function that does not depend on $\omega_{\left(u,  \, \phi,  \, \zeta'\right)}$, $\chi$, or $\zeta$. 
\end{proof}

Proposition $\ref{Proposition5}$ now gives us enough information to distinguish between middle supercuspidal representations. Let $a_{i} := \frac{u}{c_{i}u^{2}+d_{i}u-1}$.

\begin{proposition}\label{Proposition6}
    Suppose $\gamma(s, \pi_{(f_{1},  \, \chi_{1},  \, \zeta_{1})} \times \tau, \psi_{F}) = \gamma(s, \pi_{(f_{2}, \, \chi_{2},  \, \zeta_{2})} \times \tau, \psi_{F})$ for any $\tau \in \Xi_{\text{middle}}$, $\tau$ a tamely ramified quasi-character, or $\tau$ a simple supercuspidal representation of $\GL(N, F)$. Then $\pi_{(f_{1},  \, \chi_{1},  \, \zeta_{1})} \cong \pi_{(f_{2}, \, \chi_{2},  \, \zeta_{2})}$. 
\end{proposition}
\begin{proof}
    From Corollary \ref{CorollaryCentral}, if $\pi_{(f_{1},  \, \chi_{1},  \, \zeta_{1})}$ and $\pi_{(f_{2}, \, \chi_{2},  \, \zeta_{2})}$ share the same twisted gamma factors with any $\tau \in \Xi_{\text{middle}}$, then $\omega_{(f_{1},  \, \chi_{1},  \, \zeta_{1})} = \omega_{(f_{2}, \, \chi_{2},  \, \zeta_{2})}$. Furthermore, Proposition \ref{Proposition3} tells us that $\zeta_{1} = \zeta_{2}$ and $c_{1} \equiv c_{2} \mod{\mathcal{P}_{F}}$. 

    Suppose $\gamma(s, \pi_{(f_{1}, \, \chi_{1}, \, \zeta)} \times \pi_{\left(u,  \, \phi,  \, \zeta'\right)}, \psi_{F}) = \gamma(s, \pi_{(f_{2}, \, \chi_{2}, \, \zeta)} \times \pi_{\left(u,  \, \phi,  \, \zeta'\right)}, \psi_{F})$ for all simple supercuspidal representations $\pi_{\left(u,  \, \phi,  \, \zeta'\right)}$. Then Proposition \ref{Proposition5} tells us that   
    \begin{align}\label{MiddleContradiction}
       & \chi_{1}(a_{1}cu + a_{1}\sigma_{f_{1}}) \, \omega_{\left(u,  \, \phi,  \, \zeta'\right)}(a_{1}u\varpi_{F}^{2})  P_{1}\left(q_{F}^{-s}\right)  \notag \\
        & = \chi_{2}(a_{2}cu + a_{2}\sigma_{f_{2}}) \, \omega_{\left(u,  \, \phi,  \, \zeta'\right)}(a_{2}u\varpi_{F}^{2})  P_{2}\left(q_{F}^{-s}\right)
    \end{align}
    where $P_{i}$ is a non-zero rational function in $\mathbb{C}\left(q_{F}^{-s}\right)$ that does not depend on $\omega_{\left(u,  \, \phi,  \, \zeta'\right)}$ or $\chi_{i}$. Set $v := a_{1}a_{2}^{-1}$.
     
    Equation (\ref{MiddleContradiction}) implies $\omega_{\left(u,  \, \phi,  \, \zeta'\right)}(v) = P\left(q_{F}^{-s}\right)$ where $P\left(q_{F}^{-s}\right)$ is a non-zero rational function that does not depend on $\omega_{\left(u,  \, \phi,  \, \zeta'\right)}$. Fix some $s' \in \mathbb{C}$. We recall that the central character $\omega_{\left(u,  \, \phi,  \, \zeta'\right)}$ of a simple supercuspidal representation when restricted to $\mathcal{O}_{F}^{\times}$ is the inflation of a quasi-character of $k_{F}^{\times}$. Since $\omega_{\left(u,  \, \phi,  \, \zeta'\right)}$ is allowed to be any central character, we have that $\phi(\overline{v}) = P\left(q_{F}^{-s'}\right)$ for all quasi-characters $\phi$ of $k_{F}^{\times}$ where $\overline{v}$ is the image of $v$ in $k_{F}^{\times}$. This implies $v \in 1 + \mathcal{P}_{F}$ and $d_{1} \equiv d_{2} \mod{\mathcal{P}_{F}}$. 

    If $d_{1} \equiv d_{2} \mod{\mathcal{P}_{F}}$, then $P_{1} = P_{2}$. From (\ref{MiddleContradiction}),  
    \[
    \chi_{1}\left(1 + (cu)^{-1}\sigma_{f}\right) = \chi_{2}\left(1 + (cu)^{-1}\sigma_{f}\right)
    \]
    for all $u \in \mathcal{O}_{F}^{\times}$. Hence, $\chi_{1} = \chi_{2}$ as $\mathcal{O}_{L_{f}}^{\times}$ is generated by elements of the form $1 + u'\sigma_{f}$ with $u' \in \mathcal{O}_{F}^{\times}$. 
\end{proof}

\section{Alternate supercuspidal representations of depth $\frac{1}{N}$}\label{SectionAlternate}

In this section, we use twisted gamma factors to distinguish middle supercuspidals from other depth $\frac{1}{N}$ supercuspidal representations of $\GL(2N, F)$. We use conductor formulas in \cite{BHK} of $f(\tau_{1} \times \tau_{2}, \psi_{F})$ attached to a pair of irreducible supercuspidal representations $\tau_{1}$ and $\tau_{2}$ via the notion of supercuspidal representations being \textit{completely distinct}. To use these conductor formulas, we recall the characteristic polynomial $\phi_{Y_{\beta}}(X) \in k_{F}[X]$ of a simple stratum defined in subsection \ref{MaximalSimpleTypes}. We further recall that $\phi_{Y_{\beta}}(X)$ is a power of the minimal polynomial $\phi_{\pi}(X)\in k_{F}[X]$ of the simple stratum which is a monic irreducible polynomial such that $\phi_{\pi}(X) \neq X$.

\subsection{Conductor formulas}\label{ConductorFormulas}

We begin this subsection with the following definition of irreducible supercuspidal representations being \textit{completely distinct}:
\begin{definition}
    \normalfont Let $\pi_{1}$ and $\pi_{2}$ be irreducible supercuspidal representations of $\GL(n, F)$. We say that the $\pi_{i}$ are \textit{completely distinct} if either $d(\pi_{1}) \neq d(\pi_{2})$ or $d(\pi_{1}) = d(\pi_{2})$ but $\phi_{\pi_{1}} \neq \phi_{\pi_{2}}$. 
\end{definition}

Let us consider all possible depth $\frac{1}{N}$ supercuspidal representations of $\GL(2N, F)$ that are not minimax, i.e. not middle supercuspidal representations, and compute the characteristic polynomials of their corresponding simple strata. We refer to such supercuspidal representations as \textit{supercuspidal representations not of middle type}.
\begin{enumerate}
    \item Let $\rho$ be a depth $\frac{1}{N}$ supercuspidal representation of $\GL(2N, F)$ induced from the simple stratum $\left[\mathfrak{A}_{2}, 1, 0, \beta \right]$ where $E = F\left[\beta\right]$ is a degree $N$ totally ramified extension. Then the characteristic polynomial $\phi_{Y_{\beta}}(X)$ of $\left[\mathfrak{A}_{2}, 1, 0, \beta \right]$ is the characteristic polynomial of $Y_{\beta} = \beta^{N}\varpi_{F} \in A$ with coefficients reduced modulo $\mathcal{P}_{F}$. Hence, $\phi_{\rho}(X)$ is a linear polynomial $X - u$ for some $u \in k_{F}^{\times}$.
    
    \item Let $\rho'$ be a depth $\frac{1}{N}$ supercuspidal representation of $\GL(2N, F)$ induced from the simple stratum $\left[\mathfrak{I}_{2N}, 2, 0, \beta' \right]$ where $E' = F\left[\beta'\right]$ is a degree $2N$ totally ramified extension. Since $\beta'$ here is not minimal over $F$, we have a non-trivial defining sequence associated to $\left[\mathfrak{I}_{2N}, 2, 0, \beta' \right]$:
    \[
    \left[\mathfrak{I}_{2N}, 2, 0, \beta' \right] \to \left[\mathfrak{I}_{2N}, 2, 1, \beta_{1}'\right]
    \]
    where $F\left[\beta'_{1}\right]$ is a degree $N$ totally ramified extension with $\beta'_{1}$ minimal over F and $\left[\mathfrak{I}_{2N}, 2, 1, \beta_{1}'\right]$ fundamental.
    
    Then the characteristic polynomial $\phi_{Y_{\beta'}}(X)$ of $\left[\mathfrak{I}_{2N}, 2, 0, \beta' \right]$ is the characteristic polynomial of $Y_{\beta_{1}'}=(\beta_{1}')^{N}\varpi_{F} \in A$ with coefficients reduced modulo $\mathcal{P}_{F}$. Hence, $\phi_{\rho'}(X)$ is a linear factor of the form $X - u$ for some $u \in k_{F}^{\times}$.
\end{enumerate}

Next, we prove the following lemma which separates middle supercuspidals from supercuspidals not of middle type in terms of complete distinction with simple supercuspidals of $\GL(N, F)$.
\begin{lemma}\label{Lemma4}
    Let $\rho$ be a depth $\frac{1}{N}$ supercuspidal representation of $\GL(2N, F)$ that is not of middle type. Then there exists a simple supercuspidal representation $\pi_{(u,  \, \phi,  \, \zeta')}$ of $\GL(N, F)$ such that $\pi_{(u,  \, \phi,  \, \zeta')}$ is not completely distinct from $\rho$. On the other hand, $\pi_{(f,  \, \chi,  \, \zeta)}$ is completely distinct from $\pi_{(u,  \, \phi,  \, \zeta')}$ for all $u \in \mathcal{O}_{F}^{\times}$. 
\end{lemma}
\begin{proof}
    We first compute $\phi_{\pi_{(f,  \, \chi,  \, \zeta)}}$ and $\phi_{\pi_{(u,  \, \phi,  \, \zeta')}}$.
    \begin{enumerate}
    \item $\phi_{\pi_{(f,  \, \chi,  \, \zeta)}}$: From the proof of Proposition \ref{Bijection}, we see that $\phi_{\pi_{(f,  \, \chi,  \, \zeta)}}(X) = \bar{f}$ which is a monic degree two irreducible polynomial. 
    
    \item $\phi_{\pi_{(u,  \, \phi,  \, \zeta')}}$: Associated with $\pi_{(u,  \, \phi,  \, \zeta')}$ is the simple stratum $\left[\mathfrak{I}_{N}, 1, 0, \beta_{u}\right]$ given in subsection \ref{SimpleSupercuspidalRepresentations} which is fundamental from \cite[(2.3.2)]{BK}. Hence, we may set 
    \[
    Y_{\beta_{u}} = \beta_{u}^{N}\varpi_{F} \in \Mat_{N \times N}(F).
    \]
    This implies $\phi_{\pi_{(u,  \, \phi,  \, \zeta')}}(X) = X-u^{-1}$ when we reduce the characteristic polynomial of $Y_{\beta_{u}}$ modulo $\mathcal{P}_{F}$ (here we view $u^{-1}$ via its image under $\mathcal{O}_{F}^{\times} \twoheadrightarrow k_{F}^{\times}$). 
\end{enumerate}

The computation above implies the second assertion. For the first assertion, the discussion preceding Lemma \ref{Lemma4} tells us that the characteristic polynomial of $\rho$ is a power of a monic linear term $X-v$ with $v \in k_{F}^{\times}$. Choosing the simple supercuspidal representation $\pi_{(v^{-1}, \, \phi, \, \zeta)}$, we complete the proof. 
\end{proof}

From Lemma \ref{Lemma4} and noting that supercuspidals not of middle type are not unramified twists of $\pi_{(u,  \, \phi,  \, \zeta')}$, we obtain the following proposition. 
\begin{proposition}\label{Proposition8}
    Let $\rho$ be a supercuspidal representation not of middle type and $\pi_{(u,  \, \phi,  \, \zeta')}$ a simple supercuspidal representation of $\GL(N, F)$ not completely distinct from $\rho$. Then 
    \[
    \gamma\left(s, \pi_{(f,  \, \chi,  \, \zeta)} \times \pi_{(u,  \, \phi,  \, \zeta')}^{\vee}, \psi_{F}\right) \neq \gamma\left(s, \rho \times \pi_{(u,  \, \phi,  \, \zeta')}^{\vee}, \psi_{F}\right).
    \] 
\end{proposition}
\begin{proof}
From Theorem \ref{Conductor} in subsection \ref{TwistedLocalFactors}, we may rewrite the twisted gamma factors $\gamma\left(s, \pi_{(f,  \, \chi,  \, \zeta)} \times \pi_{(u,  \, \phi,  \, \zeta')}^{\vee}, \psi_{F}\right)$ and $\gamma\left(s, \rho \times \pi_{(u,  \, \phi,  \, \zeta')}^{\vee}, \psi_{F}\right)$ in terms of the conductors $f\left(\pi_{(f,  \, \chi,  \, \zeta)} \times \pi_{(u,  \, \phi,  \, \zeta')}^{\vee}\right)$ and $f\left(\rho \times \pi_{(u,  \, \phi,  \, \zeta')}^{\vee}\right)$ respectively. Since $\pi_{(u,  \, \phi,  \, \zeta')}$ is not completely distinct from $\rho$ but is completely distinct from $\pi_{(f,  \, \chi,  \, \zeta)}$, we get from \cite[Theorem 6.5 (ii), (iii)]{BHK} that
\[
f\left(\rho \times \pi_{(u,  \, \phi,  \, \zeta')}^{\vee}\right) < 2N^{2}\left(1+ \frac{1}{N}\right) = f\left(\pi_{(f,  \, \chi,  \, \zeta)} \times \pi_{(u,  \, \phi,  \, \zeta')}^{\vee}\right).
\]
Hence, this implies the proposition. 
\end{proof} 
\begin{remark}
\normalfont
    We note that since $\pi_{(u,  \, \phi,  \, \zeta')}$ is a simple supercuspidal representation of $\GL(N, F)$, its contragredient $\pi_{(u,  \, \phi,  \, \zeta')}^{\vee}$ is also a simple supercuspidal representation of $\GL(N, F)$. 
\end{remark}

\subsection{Proof of Theorem \ref{TheoremMain}}

From the results given in section \ref{SectionDistinguishing} and subsection \ref{ConductorFormulas}, we give the proof of Theorem \ref{TheoremMain} which uniquely determines middle supercuspidal representations via its twisted gamma factors. 

\begin{proof}[Proof of Theorem \ref{TheoremMain}]
    Let $\pi$ be a middle supercuspidal representation and $\pi'$ an irreducible supercuspidal representation of $\GL(2N, F)$ such that their twisted gamma factors satisfy the hypothesis of Theorem \ref{TheoremMain}. From Proposition \ref{Xu}, we have $d(\pi') = \frac{1}{N}$. The proof now follows from Propositions \ref{Proposition6} and \ref{Proposition8}.
\end{proof}

 \newcommand{\etalchar}[1]{$^{#1}$}

\end{document}